\documentclass[11pt]{amsart}

\usepackage[a4paper,margin=1.05in]{geometry}
\usepackage{amsmath,amssymb,amsthm,mathtools}
\usepackage[expansion=false]{microtype}
\usepackage[hidelinks]{hyperref}
\usepackage[nameinlink,capitalize]{cleveref}
\usepackage{color}

\numberwithin{equation}{section}
\allowdisplaybreaks

\newtheorem{theorem}{Theorem}[section]
\newtheorem{proposition}[theorem]{Proposition}
\newtheorem{lemma}[theorem]{Lemma}

\theoremstyle{definition}

\theoremstyle{remark}

\newcommand{\Sym}{\operatorname{Sym}}
\newcommand{\tr}{\operatorname{tr}}
\newcommand{\co}{\operatorname{co}}
\newcommand{\R}{\mathbb R}
\newcommand{\cQ}{\mathcal Q}
\newcommand{\cR}{\mathcal R}
\newcommand{\cF}{\mathcal F}

\title{Interior \(C^{2,\alpha}\) Regularity for Convex Solutions of the
2-Hessian Equation}
\author{Xingchen Zhou \& Ruixuan Zhu}

\address{Xingchen Zhou, 
School of Mathematics and Statistics, Hainan University, Haikou, 570228, PR China.}
\email{zxc3zxc4zxc5@stu.xjtu.edu.cn}

\address{Ruixuan Zhu, 
Institute for Theoretical Sciences, Westlake University, Hangzhou, 310030, China.}
\email{zhuruixuan@westlake.edu.cn}

\date{}

\begin{document}

\begin{abstract}
We consider convex admissible viscosity solutions of
\(\sigma_2(D^2u)=f>0\) in an open subset of \(\mathbb R^n\), where
\(n\ge2\), \(0<\alpha<1\), and
\(f\in C_{\mathrm{loc}}^{0,\alpha}\).
We prove that every such solution belongs to
\(C_{\mathrm{loc}}^{2,\alpha}\).
For solutions in \(B_2\), we also prove a uniform \(C^{2,\alpha}(B_{1/4})\) estimate under an \(L^\infty\) bound for \(u\), a positive lower bound for \(f\), and a \(C^{0,\alpha}\) bound for \(f\).
\end{abstract}


\maketitle

\section{Introduction}

In this paper, we study the interior Schauder estimates for the convex admissible solutions of the 2-Hessian equation
\[
 \sigma_2(D^2u)=f(x) \quad\text{ in } \Omega\subset \mathbb{R}^n.
\]
where we assume that $f$ is positive and in $C^{\alpha}(\Omega)$. Our main result is stated as follows.
\qquad
\begin{theorem}
\label{main}
Let  \(n\ge2\), \(0<\alpha<1\) be two constants, and let
\(\Omega\subset\mathbb R^n\) be open.  Suppose that \( u\in C(\Omega)\)
is convex and is an admissible viscosity solution of
\[
 \sigma_2(D^2u)=f(x)
 \qquad\text{in }\Omega,
\]
where
\[
 f\in C_{\mathrm{loc}}^{0,\alpha}(\Omega),
 \qquad
 f>0\quad\text{in }\Omega.
\]
Then
\[
 u\in C_{\mathrm{loc}}^{2,\alpha}(\Omega).
\]
If, in addition, \(\Omega=B_2\) and
\[
 u\in L^\infty(B_2),\qquad
 f^{-1}\in L^\infty(B_2),\qquad
 f\in C^{0,\alpha}(B_2),
\]
then
\begin{equation}
 \|u\|_{C^{2,\alpha}(B_{1/4})}
 \le
 C\!\left(
 n,\alpha,\|u\|_{L^\infty(B_2)},
 \|f^{-1}\|_{L^\infty(B_2)},
 \|f\|_{C^{0,\alpha}(B_2)}
 \right).
 \label{main-est}
\end{equation}
\end{theorem}

Notice that when \(n=2\), the equation is the
Monge--Amp\`ere equation, and the conclusion follows from the interior regularity
theory of Caffarelli and the continuity estimates of Jian and Wang
\cite{CaffarelliMA,JianWang}.

\vspace{0.7\baselineskip}

We first introduce the classical results for $k$-Hessian
equations.  Caffarelli, Nirenberg, and Spruck  \cite{CNS} established the classical
Dirichlet theory on the admissible cone.  Urbas \cite{Urbas1990}
constructed none $C^{1,1}$ admissible solutions for \(\sigma_k\) when \(k\ge3\)
, thus general interior
regularity fails when $k\ge 3$. Chou and Wang \cite{ChouWang} proved a Pogorelov estimate in
the presence of a strict \(k\)-convex barrier function. 

\vspace{0.7\baselineskip}

For smooth solutions of the \(\sigma_2\) equation, Warren and Yuan  \cite{WarrenYuan} used the special Lagrangian structure to derive Hessian estimates 
for the constant right-hand side in
dimension three. Later Qiu \cite{Qiu} treated the
$C^2$ variable right-hand sides case, and Xu
proved a H\"older estimate when the $C^\alpha$ norm of the
right-hand side is small \cite{Xu}.  In dimension four, Shankar and
Yuan \cite{ShankarYuanFour} recently obtained the unconditional estimate for the constant equation
and derived regularity for viscosity solutions
.  Fan extended the four-dimensional estimate
to positive \(C^{1,1}\) right-hand sides \cite{FanVariable}.  

\vspace{0.7\baselineskip}

In general dimensions, Guan and Qiu \cite{GuanQiu} proved Hessian estimates for
smooth solutions under a convex  restriction
\(\sigma_3\).  McGonagle, Song, and Yuan \cite{McGonagleSongYuan} used
compactness argument for constant right-hand side and almost convex smooth
solutions.  Shankar and Yuan \cite{ShankarYuanSemiconvex} obtained
estimates for semiconvex smooth solutions. They also proved regularity for almost convex
viscosity solutions with the constant right hand side
\cite{ShankarYuanAlmostConvex}.  Li and Wu \cite{LiWu} proved the Hessian estimates in general dimensions $n\ge 2$.  However, these
a priori estimates do not automatically imply interior $C^{2,\alpha}$ regularity with $C^{\alpha}$ right hand side.

\vspace{0.7\baselineskip}

A new progress was by Mooney \cite{Mooney}, which proved strict \(2\)-convexity and
interior smoothness for convex solutions with constant right-hand side
.  Very recently,  Chen, Jian, Tu, and Zhou \cite{ChenJianTuZhou} developed a  boundary Jacobi approach to
established interior \(C^2\) regularity for convex viscosity solutions
with positive Lipschitz right-hand side, with the help of Mooney's 2-convex barrier
.  In view of Wang's none \(C^{1,1}\)  viscosity solution  to Monge–Amp\`ere equation which is automatically extended to $\sigma_2$ equations, the interested open case is when
\[
 f\in C^\alpha,\qquad 0<\alpha<1.
\]

The proof uses ideas from two related lines of work.  Legendre--Lewy
transforms were used in the quadratic Hessian estimates of
McGonagle--Song--Yuan and Shankar--Yuan
\cite{ChangYuan,McGonagleSongYuan,ShankarYuanSemiconvex}.  Partial Legendre
transforms and constant-rank arguments provide precedents for the
degeneracy analysis below
\cite{GuanPhong,CaffarelliGuanMa,SzekelyhidiWeinkove}.  For the local
regularity step, we use Fan's nonhomogeneous small-perturbation theorem
\cite{Fan}, which extends Savin's theorem \cite{Savin}.  Related
pointwise criteria for locally uniformly elliptic equations were
developed by Lian and Zhang under explicit smallness assumptions
\cite{LianZhang}.

\vspace{0.7\baselineskip}

There are two main difficulties.  First, the linearized operator of
\(\sigma_2\) need not be uniformly elliptic before \(D^2u\) is
bounded.  Thus neither Evans--Krylov theory nor a standard Schauder
estimate can be applied directly to the original equation.  Second,
the Legendre transform used to restore uniform ellipticity may have a
singular Hessian.  H\"older continuity of that Hessian only does not permit inversion of the transform.

\vspace{0.7\baselineskip}

We resolve the first difficulty locally.  At an interior point, we
subtract a supporting affine function from \(u\), add
\(\kappa|x|^2/2\), and take the Legendre transform on a small ball.
The resulting function \(w\) satisfies
\[
 w\in C^{1,1},\qquad 0\le D^2w\le\kappa^{-1}I.
\]
The \(\sigma_2\) equation becomes a Hessian quotient equation for
\(w\).  Its Hessians lie in a compact subset of \(\Gamma_{n-1}\),
where \(\sigma_{n-1}/\sigma_{n-2}\) is concave and strictly elliptic.
We extend this quotient to a concave uniformly elliptic operator on
all symmetric matrices.  After reducing the ball, the
oscillation of \(f\) is small enough for Fan's perturbation theorem,
and we obtain \(w\in C^{2,\alpha}\). 

\vspace{0.7\baselineskip}

It remains to rule out a zero eigenvalue of \(D^2w\).  The dual
equation first forces any singular Hessian to have a one-dimensional
kernel, with all other eigenvalues equal to \(\kappa^{-1}\).  
Two changes of variables isolate the degenerate direction and produce
a nonnegative function \(Z\).  This function vanishes exactly when
the dual Hessian is singular and satisfies a linear adjoint equation.  
The Hölder continuity of \(f\) gives the coefficient a Dini modulus
at the base point.  The strong maximum principle then shows that if
\(Z\) vanishes there, it vanishes in a whole neighborhood.
The resulting splitting of \(w\) produces,
by Fenchel duality, an \((n-1)\)-dimensional contact set between \(u\)
and a supporting affine function.  Mooney's strict \(2\)-convexity
theorem excludes such a contact set.  Hence \(D^2w>0\), and the inverse
Legendre transform gives \(u\in C^{2,\alpha}\). 

\vspace{0.7\baselineskip}

The qualitative proof does not give a lower bound for the smallest
eigenvalue of \(D^2w\) in terms of the data.  In Section~6, compactness
converts the strict positivity at the center of each localized dual
transform into a uniform lower bound.  The second-order duality
relation then gives a uniform bound for \(D^2u\) on \(B_{1/2}\).  On
this bounded Hessian range, \(\sqrt{\sigma_2}\) is concave and
uniformly elliptic.  The standard interior Schauder estimate completes
\eqref{main-est}.

\vspace{0.7\baselineskip}

The paper is organized as follows.  Section~2 constructs the localized
Legendre transform, derives the dual equation, and extends it to a
uniformly elliptic viscosity equation.  Section~3 proves
\(C^{2,\alpha}\) regularity of the dual function using Fan's
small-perturbation theorem.  Section~4 rules out singular dual
Hessians by a strong maximum principle and Mooney's strict
\(2\)-convexity theorem.  Section~5 inverts the Legendre transform and
completes the qualitative proof.  Finally, Section~6 uses compactness
to obtain a uniform lower bound for the dual Hessian and proves the
quantitative estimate.

\section{The local Legendre transform}

We construct the local Legendre transform used in the proof. 
The dual function satisfies a uniformly elliptic equation almost-everywhere. By the $C^{1,1}$ regularity, it is also a viscosity solution of the equation.
This Legendre--Lewy type transformation was also used in
\cite{ChangYuan,McGonagleSongYuan,ShankarYuanSemiconvex}.

\subsection{The localized transform}

For \(M\in\Sym(n)\), let \(\sigma_k(M)\) be the \(k\)-th elementary
symmetric function of its eigenvalues, with \(\sigma_0=1\), and set
\[
 \Gamma_k
 =
 \{M\in\Sym(n):\sigma_j(M)>0\text{ for }1\le j\le k\}.
\]
A function \(\varphi\in C^2\) is admissible at \(x_0\) if
\(
D^2\varphi(x_0)\in\Gamma_2.
\)
A continuous function \(u\) is a viscosity subsolution of
\(
\sigma_2(D^2u)=f
\)
if every admissible function \(\varphi\) touching \(u\) from above at
\(x_0\) satisfies
\(
\sigma_2(D^2\varphi(x_0))\ge f(x_0).
\)
It is a viscosity supersolution if every admissible function
\(\varphi\) touching \(u\) from below at \(x_0\) satisfies
\(
\sigma_2(D^2\varphi(x_0))\le f(x_0).
\)
An admissible viscosity solution is both a subsolution and a
supersolution.

\begin{lemma}
\label{alex}
Let \(u\) be a convex viscosity solution of
\[
 \sigma_2(D^2u)=f
\]
in an open set, where \(f\) is continuous and positive.  If \(u\) is
twice differentiable at \(x_0\), with Alexandrov Hessian \(A\), then
\[
 A\in\Gamma_2,
 \qquad
 \sigma_2(A)=f(x_0).
\]
\end{lemma}

\begin{proof}
Let \(p=Du(x_0)\), then we have the Alexandrov expansion
\[
 u(x_0+h)
 =
 u(x_0)+p\cdot h+\frac12h^TAh+o(|h|^2).
\]
For every \(\varepsilon>0\), after restricting to a sufficiently
small ball, the two polynomials
\[
 P_\varepsilon^\pm(x_0+h)
 =
 u(x_0)+p\cdot h
 +\frac12h^T(A\pm\varepsilon I)h
\]
touch \(u\) from above and below, respectively, at \(x_0\).
The upper test and the viscosity subsolution inequality give
\[
 \sigma_2(A+\varepsilon I)\ge f(x_0).
\]
Letting \(\varepsilon\downarrow0\), we obtain
\(\sigma_2(A)\ge f(x_0)>0\).  Since \(u\) is convex, \(A\ge0\), and hence
\(\sigma_1(A)>0\) and \(A\in\Gamma_2\).  Since \(\Gamma_2\) is open,
\(A-\varepsilon I\in\Gamma_2\) for every sufficiently small
\(\varepsilon>0\).  The lower test is then admissible.  By the
viscosity supersolution inequality,
\[
 \sigma_2(A-\varepsilon I)\le f(x_0).
\]
Letting \(\varepsilon\downarrow0\) proves
\(\sigma_2(A)=f(x_0)\).
\end{proof}

We now fix the local setting used in the qualitative proof.  Let
\(u\in C(B_{2R})\) be a convex viscosity solution of
\[
 \sigma_2(D^2u)=f(x)
 \qquad\text{in }B_{2R},
\]
where
\begin{equation}
 f\in C^{0,\alpha}(B_{2R}),
 \qquad
 0<m\le f\le M.
 \label{fbounds}
\end{equation}
After subtracting a supporting affine function, we can assume
\begin{equation}
 u(0)=0,\qquad 0\in\partial u(0),\qquad u\ge0.
 \label{norm}
\end{equation}
Let
\begin{equation}
 c=\min_{\overline{B_R}}f,
 \qquad
 \binom n2\kappa^2=c,
 \qquad
 \Phi(x)=u(x)+\frac{\kappa}{2}|x|^2,
 \label{kappa}
\end{equation}
and define
\begin{equation}
 w(y)
 =
 \sup_{x\in\overline{B_R}}
 \{x\cdot y-\Phi(x)\}.
 \label{dualdef}
\end{equation}
We choose $R$ small such that
\begin{equation}
 R^\alpha[f]_{C^{0,\alpha}(B_{2R})}\le\varepsilon_*.
 \label{small}
\end{equation}
Here
\(\varepsilon_*=\varepsilon_*(n,\alpha,m,M)>0\) will be given in Section~\ref{prereg}.

\begin{proposition}
\label{locdual}
For every \(y\in B_{\kappa R/4}\), the supremum in
\eqref{dualdef} has a unique maximizer
\(x(y)\in B_R\).  Moreover, \(w\in C^{1,1}(B_{\kappa R/4})\) and
\begin{equation}
 Dw(y)=x(y),\qquad
 \operatorname{Lip}(Dw)\le\kappa^{-1},
 \qquad
 0\le D^2w\le\kappa^{-1}I
 \quad\text{a.e. in }B_{\kappa R/4}.
 \label{C11}
\end{equation}
In particular,
\begin{equation}
 Dw(B_{\kappa R/4})\subset B_R.
 \label{image}
\end{equation}
\end{proposition}

\begin{proof}
If \(x\in\partial B_R\) and \(|y|<\kappa R/4\), then by convexity and
\eqref{norm},
\[
 x\cdot y-\Phi(x)
 \le R|y|-\frac{\kappa R^2}{2}
 <-\frac{\kappa R^2}{4},
\]
whereas the value at \(x=0\) is \(0\).  Thus the maximizer in
\eqref{dualdef} lies in \(B_R\).  Since \(\Phi\) is
\(\kappa\)-strongly convex, this maximizer is unique. We denote it by
\(x(y)\). Thus \(w\) is differentiable and \(Dw(y)=x(y)\).  If
\(x_i=x(y_i)\), strong monotonicity of \(\partial\Phi\) gives
\[
 (y_1-y_2)\cdot(x_1-x_2)
 \ge\kappa|x_1-x_2|^2.
\]
It follows that \(|x_1-x_2|\le\kappa^{-1}|y_1-y_2|\).  Hence
\eqref{C11} and \eqref{image} follow.
\end{proof}

\subsection{The dual Hessian and its equation}

When \(n=2\), all formulas below are read with
\(\sigma_0=1\); in particular
\(\cQ=\sigma_1=\tr\) and \(\cR=\sigma_2=\det\).

We next derive the dual equation.
On \(\Gamma_{n-1}\), define
\begin{equation*}
 \cQ(N)=\frac{\sigma_{n-1}(N)}{\sigma_{n-2}(N)},
 \qquad
 \cR(N)=\frac{\sigma_n(N)}{\sigma_{n-2}(N)},
 \qquad
 a(y)=f(Dw(y))-c.
\end{equation*}

\begin{proposition}
\label{dualhess}
At an Alexandrov point \(y\in B_{\kappa R/4}\), If
\(N=D^2w(y)\) is singular, it must have the form
\begin{equation}
 N=\kappa^{-1}(I-e\otimes e)
 \label{sing}
\end{equation}
for some unit vector \(e\).  Moreover,
\begin{equation}
 \kappa(n-1)\cQ(D^2w)-1
 +a(y)\cR(D^2w)=0
 \quad\text{a.e. in }B_{\kappa R/4}.
 \label{dual-ae}
\end{equation}
There is a compact set \(\mathcal K\), depending only on \(n,m,M\),
such that
\begin{equation}
 \mathcal K\Subset\Gamma_{n-1},
 \qquad
 D^2w(y)\in\mathcal K
 \quad\text{for a.e. }y\in B_{\kappa R/4}.
 \label{K}
\end{equation}
\end{proposition}

\begin{proof}
Let \(w\) be twice differentiable at
\(y_0\in B_{\kappa R/4}\), and let
\(
 N=D^2w(y_0),
 x_0=Dw(y_0).
\)
Suppose that \(N\) is singular.  We first prove
\eqref{sing}.

Indeed, for every
\(\varepsilon>0\), on a sufficiently small ball about \(y_0\), we have
\begin{equation*}
 w(y)
 \le
 w(y_0)+x_0\cdot(y-y_0)
 +
 \frac12(y-y_0)^T(N+\varepsilon I)(y-y_0)
 =:P_\varepsilon(y).
\end{equation*}
Denote that small closed dual ball by \(\overline{\mathcal U}\), and
define the restricted conjugate
\[
 Q_\varepsilon(x)
 =
 \sup_{y\in\overline{\mathcal U}}
 \{x\cdot y-P_\varepsilon(y)\}.
\]
For \(x\) near \(x_0\), its maximizer is interior, and
\(Q_\varepsilon\) is the quadratic function with Hessian
\((N+\varepsilon I)^{-1}\).  Since
\(P_\varepsilon\ge w\) on the ball,
\[
 Q_\varepsilon(x)
 \le
 \sup_{y\in\overline{\mathcal U}}
 \{x\cdot y-w(y)\}
 \le
 \sup_{y\in\R^n}\{x\cdot y-w(y)\}
 =\Phi(x)
\]
for \(x\) near \(x_0\), and equality holds at \(x_0\). Hence
\(
 Q_\varepsilon(x)-\frac{\kappa}{2}|x|^2
\)
is a lower test for \(u\) at \(x_0\), with Hessian
\begin{equation*}
 A_\varepsilon=(N+\varepsilon I)^{-1}-\kappa I.
\end{equation*}

Let \(k=\dim\ker N\).  If \(k\ge2\), then \(A_\varepsilon\) has \(k\)
eigenvalues equal to \(\varepsilon^{-1}-\kappa\), while every
remaining eigenvalue is bounded below by \(-C\varepsilon\), in view
of \(N\le\kappa^{-1}I\).  It follows that, for small \(\varepsilon\), $A_\varepsilon\in\Gamma_2$, and
\[
 \sigma_2(A_\varepsilon)
 =
 \binom{k}{2}\varepsilon^{-2}+O(\varepsilon^{-1})
 \longrightarrow+\infty.
\]
This contradicts the viscosity supersolution inequality and
\eqref{fbounds}.

Therefore \(k=1\).  In diagonal coordinates write
\[
 N=\operatorname{diag}(0,\mu_2,\ldots,\mu_n),
 \qquad
 b_j=\mu_j^{-1}-\kappa\ge0.
\]
If \(B=\sum_{j=2}^n b_j>0\), then $A_\varepsilon\in\Gamma_2$, and
\[
 \sigma_2(A_\varepsilon)
 =
 (\varepsilon^{-1}-\kappa)
 \sum_{j=2}^n
 \bigl((\mu_j+\varepsilon)^{-1}-\kappa\bigr)
 +
 O(1)
 \longrightarrow+\infty,
\]
again a contradiction.  Hence \(B=0\), and every
\(b_j\) vanishes.  This proves \eqref{sing}.

We next prove the almost-everywhere equation
\eqref{dual-ae}.

At an Alexandrov point \(y\) where \(N=D^2w(y)>0\), set
\(x=Dw(y)\).  We claim that \(g=w^*\) is twice differentiable at
\(x\), with
\[
D^2g(x)=N^{-1}.
\]
First, \(y\) is the unique point in \(\partial g(x)\).  Otherwise,
\(w\) would be affine on a nontrivial segment with endpoint \(y\).
This contradicts the positive definite second-order expansion of
\(w\) at \(y\).
Since \(x\in B_R\), the subgradients of \(g\) are locally bounded near
\(x\).  Since the graph of \(\partial g\) is closed, it follows that every
point in \(\partial g(x')\) is close to \(y\) whenever \(x'\) is close
to \(x\).
Let \(P_\pm\) be the tangent quadratic polynomials to \(w\) at \(y\),
with Hessians \(N\pm\varepsilon I\).
The preceding localization and
\(P_-\le w\le P_+\) near \(y\) give
\[
P_+^*(x')\le g(x')\le P_-^*(x')
\]
for \(x'\) close to \(x\).  Since
\[
D^2P_\pm^*=(N\pm\varepsilon I)^{-1},
\]
letting \(\varepsilon\downarrow0\) yields \(D^2g(x)=N^{-1}\).
Finally, \(g=\Phi=u+\kappa|x|^2/2\) near \(x\), and hence
\[
D^2u(x)=N^{-1}-\kappa I.
\]

By Lemma~\ref{alex} and
\[
 \sigma_j(N^{-1})
 =
 \frac{\sigma_{n-j}(N)}{\sigma_n(N)}
\]
we have
\[
 \begin{split}
 f(x)
 &=
 \sigma_2(N^{-1}-\kappa I)\\
 &=
 \frac{\sigma_{n-2}(N)}{\sigma_n(N)}
 -\kappa(n-1)
  \frac{\sigma_{n-1}(N)}{\sigma_n(N)}
 +\binom n2\kappa^2.
 \end{split}
\]
Using \eqref{kappa} and multiplying by
\(\sigma_n(N)/\sigma_{n-2}(N)\) proves
\eqref{dual-ae} at such a point.

If \(N\) is singular, then by the preceding classification,
\eqref{sing} holds.  We then calculate
\[
 \cQ(N)=\frac1{\kappa(n-1)},
 \qquad
 \cR(N)=0.
\]
Thus \eqref{dual-ae} also holds at every singular
Alexandrov point.  Alexandrov's theorem proves the almost-everywhere
claim.

The quotient \(\cQ\) is uniformly elliptic only while the Hessian
stays in a compact subset of \(\Gamma_{n-1}\). 
At a positive definite Alexandrov point put
\[
 A=N^{-1}-\kappa I\ge0
\]
and order its eigenvalues as
\(\lambda_1\ge\cdots\ge\lambda_n\ge0\).  Since
\[
 \lambda_1\lambda_2
 \le\sigma_2(A)=f\le M,
\]
we have \(\lambda_2\le\sqrt M\).  Therefore at least \(n-1\)
eigenvalues of \(N=(A+\kappa I)^{-1}\) are bounded below by
\[
 d=(\kappa+\sqrt M)^{-1}.
\]
Together with \eqref{C11}, we obtain
\begin{equation}
 \sigma_{n-2}(N)\ge d^{\,n-2},
 \qquad
 \sigma_{n-1}(N)\ge d^{\,n-1}.
 \label{sigmabounds}
\end{equation}
The same estimates hold for singular $N$
due to \eqref{sing}.  It remains to construct the uniform compact
set in \eqref{K}.  Set
\[
 \kappa_-=\sqrt{\frac{m}{\binom n2}},
 \qquad
 \kappa_+=\sqrt{\frac{M}{\binom n2}},
 \qquad
 d_0=(\kappa_++\sqrt M)^{-1}.
\]
Since \(\kappa_-\le\kappa\le\kappa_+\) and \(d\ge d_0\), one may take
\begin{equation*}
 \mathcal K
 =
 \left\{
 N\in\Sym(n):
 0\le N\le\kappa_-^{-1}I,\quad
 \sigma_{n-1}(N)\ge d_0^{\,n-1}
 \right\}.
\end{equation*}
This set is compact.  Every one of its elements is nonnegative and
has rank at least \(n-1\), so
\(\sigma_j(N)>0\) for \(1\le j\le n-1\).  Thus
\(\mathcal K\Subset\Gamma_{n-1}\), as asserted.
Because the G{\aa}rding cone \(\Gamma_{n-1}\) is convex, $\co\mathcal K\Subset\Gamma_{n-1}$.
\end{proof}

\subsection{A uniformly elliptic viscosity equation}

The almost-everywhere equation is elliptic only on the compact set
\(\mathcal K\).  We extend it to all symmetric matrices.

\begin{proposition}
\label{viscosity}
After choosing \(\varepsilon_*\) in
\eqref{small} sufficiently small, there exist a
globally defined concave uniformly elliptic function
\(Q:\Sym(n)\to\mathbb R\) and a function
\(\widetilde{\cR}\in C^\infty(\Sym(n))\) with bounded derivatives.
They agree with \(\cQ\) and \(\cR\), respectively, in a neighborhood
of \(\mathcal K\).  The operator
\begin{equation}
 \cF(M,y)
 =
 \kappa(n-1)Q(M)-1+a(y)\widetilde{\cR}(M)
 \label{Fext}
\end{equation}
is uniformly elliptic on \(\Sym(n)\times B_{\kappa R/4}\), and
\begin{equation}
 \cF(D^2w,y)=0
 \label{dual-visc}
\end{equation}
in the viscosity sense in $B_{\kappa R/4}$.
\end{proposition}

\begin{proof}

Choose a compact convex set
\begin{equation*}
 \co\mathcal K
 \Subset\operatorname{int}\mathcal C
 \Subset\mathcal C
 \Subset\Gamma_{n-1}.
\end{equation*}
The quotient \(\cQ\) is smooth, concave, and strictly elliptic on
\(\Gamma_{n-1}\); see, for example, \cite[Section~1]{CNS}.  Thus, on
\(\mathcal C\),
\begin{equation}
 \lambda I\le D\cQ(P)\le\Lambda I
 \label{Qell}
\end{equation}
for some \(0<\lambda\le\Lambda\).  Define
\begin{equation*}
 Q(M)
 =
 \inf_{P\in\mathcal C}
 \{\cQ(P)+D\cQ(P):(M-P)\}.
\end{equation*}
This is a finite, continuous, concave function on \(\Sym(n)\).
By the concavity of \(\cQ\), followed by the choice \(P=M\),
\[
 Q=\cQ\quad\text{on }\mathcal C.
\]
For \(H\ge0\), by \eqref{Qell} and the tangent-plane formula,
\begin{equation*}
 \lambda\tr H
 \le Q(M+H)-Q(M)
 \le\Lambda\tr H.
\end{equation*}

Choose \(\chi\in C_c^\infty(\Gamma_{n-1})\) equal to one in a
neighborhood of \(\mathcal C\), and define
\[
 \widetilde{\cR}=\chi\cR
\]
there, extending it by zero to all of \(\Sym(n)\).  It is smooth with
bounded derivative.  Therefore, if \(\|a\|_\infty\) is sufficiently
small, the operator \eqref{Fext} is uniformly
elliptic.  
Since \(D^2w\in\mathcal K\) almost everywhere, this extension does not
change the almost-everywhere equation \eqref{dual-ae}.
Indeed, for \(H\ge0\),
\[
 \begin{split}
 &\bigl(\kappa(n-1)\lambda
 -\|a\|_\infty\|D\widetilde{\cR}\|_\infty\bigr)\tr H\\
 &\quad\le
 \cF(M+H,y)-\cF(M,y)\\
 &\quad\le
 \bigl(\kappa(n-1)\Lambda
 +\|a\|_\infty\|D\widetilde{\cR}\|_\infty\bigr)\tr H.
 \end{split}
\]
The lower constant is positive once
\(\|a\|_\infty\|D\widetilde{\cR}\|_\infty
<\kappa(n-1)\lambda\).

We use the following standard fact.  Let
\(F:\Sym(n)\times\Omega\to\mathbb R\) be continuous and uniformly
elliptic.  If
\[
v\in W_{\mathrm{loc}}^{2,\infty}(\Omega)\cap C(\Omega),
\qquad
F(D^2v(x),x)=0
\quad\text{for a.e. }x\in\Omega,
\]
then \(v\) is a viscosity solution of the same equation; see
\cite[Chapters~2--3]{CaffarelliCabre}.

We now apply this fact to \(\cF\).  Recall that
\[
a(y)=f(Dw(y))-c.
\]
Since \(c=\min_{\overline{B_R}}f\) and \eqref{image} holds,
\[
0\le a(y)
\le \operatorname{osc}_{B_R}f
\le 2^\alpha R^\alpha[f]_{C^{0,\alpha}(B_R)}.
\]
Moreover, by \eqref{C11},
\[
|a(y)-a(y')|
\le
[f]_{C^{0,\alpha}(B_R)}
\kappa^{-\alpha}|y-y'|^\alpha.
\]
After choosing the constant in \eqref{small} sufficiently small,
the operator \(\cF\) is continuous and globally uniformly elliptic.
Since \(D^2w\in\mathcal K\) almost everywhere, the extensions used in
\(\cF\) agree with the original operators at \(D^2w\).  Hence
\eqref{dual-ae} gives
\[
\cF(D^2w,y)=0
\quad\text{for a.e. }y\in B_{\kappa R/4}.
\]
The standard result above now yields \eqref{dual-visc}.
\end{proof}

\section{$C^{2,\alpha}$ regularity of the dual function}
\label{prereg}

Propositions~\ref{locdual}
and~\ref{viscosity} give \(w\in C^{1,1}\) and the
viscosity equation \eqref{dual-visc}.  We now
upgrade \(w\) to \(C^{2,\alpha}\).  At this point we know only that
\(D^2w\ge0\); strict positivity will be proved in
Section~\ref{exclude}.

\begin{proposition}
\label{regularity}
Under the local assumptions and notation of Section~2, there is
a positive constant \(\varepsilon_*\), depending only on
\(n,\alpha,m,M\), such that, if \eqref{small}
holds,
\begin{equation}
 w\in C^{2,\alpha}(B_{\kappa R/8}).
 \label{C2a}
\end{equation}
Moreover,
\begin{equation}
 \|D^2w\|_{L^\infty(B_{\kappa R/8})}
 +(\kappa R)^\alpha
   [D^2w]_{C^{0,\alpha}(B_{\kappa R/8})}
 \le C,
 \label{scaled}
\end{equation}
where \(C\) depends only on \(n,\alpha,m,M\).
\end{proposition}

To improve the regularity of \(w\) from \(C^{1,1}\) to \(C^{2,\alpha}\), we apply a
local small perturbation theorem from
\cite[Theorem~1.7]{Fan}.  

 \begin{lemma}
\label{fan}
Let \(0<\alpha<1\) and \(\rho>0\).  Let
\[
 G:\Sym(n)\times B_1\longrightarrow\R
\]
be continuous and globally
\((\lambda_0,\Lambda_0)\)-uniformly elliptic, and suppose
\[
 G(0,x)=0
 \qquad(x\in B_1).
\]
Assume that \(M\mapsto G(M,x)\) is \(C^1\) for
\(\|M\|<2\rho\), uniformly in \(x\), and that \(D_MG\) has a joint
modulus of continuity there: for
\(\|M\|,\|M'\|<2\rho\),
\begin{equation}
 \|D_MG(M,x)-D_MG(M',x')\|
 \le
 \omega_G\bigl(\|M-M'\|+|x-x'|\bigr).
 \label{Fan-mod}
\end{equation}
Assume also that
\begin{equation}
 |G(M,x)-G(M,x')|
 \le\delta|x-x'|^\alpha
 \qquad
 (\|M\|<2\rho,\ x,x'\in B_1).
 \label{Fan-x}
\end{equation}
If \(z\in C(B_1)\) is a viscosity solution of
\[
 G(D^2z,x)=h(x)
 \qquad\text{in }B_1
\]
and
\begin{equation}
 \|z\|_{L^\infty(B_1)}
 +\|h\|_{C^{0,\alpha}(B_1)}
 \le\delta,
 \label{Fan-small}
\end{equation}
then, for
\[
 \delta
 =
 \delta(n,\alpha,\rho,\lambda_0,\Lambda_0,\omega_G)>0
\]
sufficiently small,
\[
 z\in C^{2,\alpha}(B_{1/2}),
 \qquad
 \|z\|_{C^{2,\alpha}(B_{1/2})}\le C.
\]
\end{lemma}

\begin{proof}
The proof of \cite[Theorem~1.7]{Fan} applies to the present operator
\(G(D^2z,x)\).  Since the operator does not depend on \(z\) or \(Dz\),
the corresponding terms are absent.  The remaining assumptions are
given by uniform ellipticity and
\eqref{Fan-mod}--\eqref{Fan-small}.
\end{proof}

\begin{lemma}
\label{perturb}
Let \(0<\alpha<1\).  Let
\(\mathcal K\subset\Sym(n)\) be compact, and let
\(\mathcal O\subset\Sym(n)\) be open with
\(\co\mathcal K\Subset\mathcal O.\)
Let \(G_0:\Sym(n)\to\R\) be globally concave and globally
\((\lambda,\Lambda)\)-uniformly elliptic.  Assume that \(G_0\) is
\(C^1\) on \(\mathcal O\), with \(DG_0\) uniformly continuous on
compact subsets of \(\mathcal O\).  Let
\(H\in C^\infty(\Sym(n))\) have bounded first derivative.

For every \(L<\infty\), there are
\(\delta>0\) and \(C<\infty\), depending only on the displayed data
and \(L\), such that the following holds.  If
\begin{equation}
 \begin{split}
 &v\in C^{1,1}(B_1),\qquad
   \|v\|_{C^{1,1}(B_1)}\le L,\qquad
   D^2v\in\mathcal K\quad\text{a.e. in }B_1,\\
 &b\in C^{0,\alpha}(B_1),\qquad
   \|b\|_{C^{0,\alpha}(B_1)}\le\delta,
 \end{split}
 \label{pert-assume}
\end{equation}
and \(v\) is a viscosity solution of
\begin{equation*}
 G_0(D^2v)+b(x)H(D^2v)=0
 \qquad\text{in }B_1,
\end{equation*}
then
\[
 v\in C^{2,\alpha}(B_{1/2}),
 \qquad
 \|v\|_{C^{2,\alpha}(B_{1/2})}\le C.
\]
\end{lemma}

\begin{proof}
Choose compact convex sets
$
 \co\mathcal K
 \Subset\operatorname{int}\mathcal C_0,
 \mathcal C_0
 \Subset\operatorname{int}\mathcal C_1,
 \mathcal C_1\Subset\mathcal O,
$
and let
\[
 d_*=\operatorname{dist}(\mathcal C_0,\partial\mathcal C_1)>0.
\]
After decreasing \(\delta\), the function
\[
 F_b(M,x)=G_0(M)+b(x)H(M)
\]
is globally uniformly elliptic, with constants independent of \(b\).
Indeed, for \(P\ge0\),
\[
 |H(M+P)-H(M)|
 \le \|DH\|_\infty\tr P,
\]
so the perturbation changes the lower ellipticity constant by at most
\(\|b\|_\infty\|DH\|_\infty\).

We first find a quadratic approximation whose Hessian remains in the
prescribed set.
Fix a number \(\varepsilon>0\), to be chosen below.  By the Evans--Krylov estimate \cite{Evans,Krylov}, there are
\(\bar\alpha=\bar\alpha(n,\lambda,\Lambda)>0\) and
\(C_1<\infty\) such that every solution \(u\) of
\[
G_0(D^2u)=0\qquad\text{in }B_1
\]
with \(\|u\|_{L^\infty(B_1)}\le L\) satisfies
\[
[D^2u]_{C^{0,\bar\alpha}(B_{1/2})}\le C_1.
\]
Here \(C_1\) depends only on
\(n,\lambda,\Lambda,|G_0(0)|\), and \(L\).
Fix
\(\eta\in(0,1/4)\) so small that
\[
 C_1\eta^{\bar\alpha}\le\frac{\varepsilon}{4}.
\]
We claim that there is \(\delta_1>0\) such that, whenever
\(\|b\|_\infty\le\delta_1\), there is a quadratic polynomial \(P\),
with \(A=D^2P\), satisfying
\begin{equation}
 \|v-P\|_{L^\infty(B_\eta)}
 \le\varepsilon\eta^2,
 \qquad
 F_b(A,0)=0,
 \qquad
 \operatorname{dist}(A,\mathcal C_0)<\frac{d_*}{4}.
 \label{quad}
\end{equation}
If no such \(\delta_1\) existed for this fixed \(\eta\), there would
be \(b_j\to0\) uniformly and solutions \(v_j\) obeying
\eqref{pert-assume} for which
\eqref{quad} fails.  After passing to a
subsequence, we have
\[
 v_j\longrightarrow v_\infty
 \quad\text{in }C^1_{\mathrm{loc}}(B_1),
 \qquad
 D^2v_j\rightharpoonup^\ast D^2v_\infty
 \quad\text{in }L^\infty_{\mathrm{loc}}(B_1).
\]
Since \(D^2v_j\in\mathcal K\) almost everywhere, the weak-* limit
takes values in the convex hull of \(\mathcal K\):
\[
D^2v_\infty\in\co\mathcal K
\quad\text{a.e.}
\]
By viscosity stability,
\[
 G_0(D^2v_\infty)=0.
\]
By the Evans--Krylov theorem
\cite{Evans,Krylov}, applied to the globally concave
uniformly elliptic operator \(G_0\), and with the exponent and
constant fixed above,
\[
 v_\infty\in C^{2,\bar\alpha}(B_{1/2})
\]
with a uniform estimate.  Since \(D^2v_\infty\) is continuous,
\(
D^2v_\infty(x)\in\co\mathcal K, x\in B_1.
\)
Let \(P_\infty\) be the second-order Taylor polynomial of
\(v_\infty\) at the origin.  Then
\[
A_\infty:=D^2P_\infty=D^2v_\infty(0)\in\co\mathcal K.
\]  Then
\[
 \|v_\infty-P_\infty\|_{L^\infty(B_\eta)}
 \le C_{\rm EK}\eta^{2+\bar\alpha}
 \le\frac{\varepsilon}{4}\eta^2.
\]
Since \(G_0(A_\infty)=0\), uniform ellipticity along
\(A_\infty+tI\) provides \(t_j\to0\) such that
\[
 F_{b_j}(A_\infty+t_jI,0)=0,
 \qquad
 |t_j|\le C\|b_j\|_\infty.
\]
The polynomials
\[
 P_j=P_\infty+\frac{t_j}{2}|x|^2
\]
satisfy, for all large \(j\),
\[
 \|v_j-P_j\|_{L^\infty(B_\eta)}
 \le
 \|v_j-v_\infty\|_{L^\infty(B_\eta)}
 {}+\frac{\varepsilon}{4}\eta^2
 {}+\frac{|t_j|}{2}\eta^2
 <\varepsilon\eta^2.
\]
Moreover,
\(\operatorname{dist}(A_\infty+t_jI,\mathcal C_0)<d_*/4\), because
\(A_\infty\in\co\mathcal K\subset\mathcal C_0\) and \(t_j\to0\).
Thus \(P_j\) satisfies all of
\eqref{quad}, which is a
contradiction.

We next reduce the lemma to the local perturbation theorem.
Let
\[
 \rho=\frac{d_*}{8},
 \qquad
 z(x)=\frac{v(\eta x)-P(\eta x)}{\eta^2}.
\]
Then \(\|z\|_{L^\infty(B_1)}\le\varepsilon\) and
\[
 \widehat F(D^2z,x)=h(x)
 \qquad\text{in }B_1,
\]
where
\begin{equation*}
 \begin{split}
 \widehat F(M,x)
 &=
 F_b(A+M,\eta x)-F_b(A,\eta x),\\
 h(x)&=-F_b(A,\eta x).
 \end{split}
\end{equation*}
Thus \(\widehat F(0,x)=0\).  The Hessian buffer in
\eqref{quad} implies
\begin{equation}
 A+M\in\mathcal C_1
 \qquad\text{whenever }\|M\|<2\rho.
 \label{buffer}
\end{equation}
Consequently \(D_M\widehat F\) has a uniform modulus of continuity in
that matrix ball.  More explicitly, uniform continuity of \(DG_0\) on
\(\mathcal C_1\), together with smoothness of \(H\) on the fixed
compact neighborhood, implies
\begin{equation*}
 \begin{split}
 &\|D_M\widehat F(M,x)-D_M\widehat F(M',x')\|\\
 &\qquad\le
 \omega_0(\|M-M'\|)
 +C\|M-M'\|+C|x-x'|^\alpha
 \end{split}
\end{equation*}
for \(\|M\|,\|M'\|<2\rho\), where \(\omega_0\) is a modulus for
\(DG_0\) on \(\mathcal C_1\).  Moreover,
\begin{equation}
 \begin{split}
 |\widehat F(M,x)-\widehat F(M,x')|
 &\le
 C\eta^\alpha[b]_{C^\alpha(B_1)}
 |x-x'|^\alpha,
 \qquad \|M\|<2\rho,\\
 \|h\|_{C^\alpha(B_1)}
 &\le C\eta^\alpha[b]_{C^\alpha(B_1)}.
 \end{split}
 \label{shift-small}
\end{equation}
Here the second estimate uses \(F_b(A,0)=0\), while the first uses
\[
 \widehat F(M,x)-\widehat F(M,x')
 =
 \bigl(b(\eta x)-b(\eta x')\bigr)
 \bigl(H(A+M)-H(A)\bigr).
\]

Let \(\delta_F\) be the smallness constant in Lemma~\ref{fan}.
Choose \(\varepsilon\le\delta_F/2\).  Once \(\eta\) and
\(\delta_1\) are fixed, choose the \(\delta\) in the present lemma so
that \(\delta\le\delta_1\), the operator \(F_b\) remains uniformly
elliptic, and both bounds in \eqref{shift-small} are at most
\(\delta_F/2\).  Then
\[
\|z\|_{L^\infty(B_1)}
+\|h\|_{C^{0,\alpha}(B_1)}
\le\delta_F.
\]
Together with \eqref{buffer} and \eqref{shift-small}, this verifies
the hypotheses of Lemma~\ref{fan}.
Hence
\(
z\in C^{2,\alpha}(B_{1/2})
\)
with a uniform estimate.  After scaling back, we obtain the
corresponding estimate for \(v\) in \(B_{\eta/2}\).  
It remains to make the estimate uniform on \(B_{1/2}\).  Fix
\(r\in(0,1/4)\).  For each \(x_0\in B_{1/2}\), define
\[
v_{x_0,r}(z)
=
\frac{v(x_0+rz)-v(x_0)-rDv(x_0)\cdot z}{r^2}.
\]
Then
\[
D^2v_{x_0,r}(z)=D^2v(x_0+rz)\in\mathcal K
\quad\text{a.e. in }B_1,
\]
and the \(C^{1,1}(B_1)\) norm of \(v_{x_0,r}\) is bounded by a
constant depending only on \(L\).  The rescaled coefficient
\(
b_{x_0,r}(z)=b(x_0+rz)
\)
has a uniformly small \(C^{0,\alpha}(B_1)\) norm.
Thus the preceding
argument gives the same local estimate at every \(x_0\in B_{1/2}\).
A finite covering of \(B_{1/2}\) completes the proof.
\end{proof}

\begin{proof}[Proof of Proposition~\ref{regularity}]
Fix \(y_0\in B_{\kappa R/8}\) and
\(
 s=\frac{\kappa R}{16}.
\)
Define
\begin{equation*}
 W(z)
 =
 \frac{w(y_0+sz)-w(y_0)-sDw(y_0)\cdot z}{s^2},
 \qquad z\in B_1.
\end{equation*}
Then \(B_s(y_0)\Subset B_{\kappa R/4}\), and by
\eqref{C11},
\[
 |W(z)|\le\frac1{2\kappa}|z|^2,
 \qquad
 |DW(z)|\le\frac1\kappa|z|,
 \qquad
 0\le D^2W\le\frac1\kappa I
\]
almost everywhere.  Thus \(W\) has the uniform full
\(C^{1,1}(B_1)\) bound required below.  Moreover,
\[
 D^2W\in\mathcal K\quad\text{a.e.}
\]
The rescaled coefficient \(a_s(z)=a(y_0+sz)\) satisfies
\begin{equation}
 \begin{split}
 \|a_s\|_{L^\infty(B_1)}
 &\le
 2^\alpha R^\alpha[f]_{C^{0,\alpha}(B_R)},\\
 [a_s]_{C^{0,\alpha}(B_1)}
 &\le
 [f]_{C^{0,\alpha}(B_R)}
 \kappa^{-\alpha}s^\alpha
 \le
 C R^\alpha[f]_{C^{0,\alpha}(B_R)}.
 \end{split}
 \label{rescale-small}
\end{equation}
Equation \eqref{dual-visc} becomes
\[
 \bigl(\kappa(n-1)Q(D^2W)-1\bigr)
 +a_s(z)\widetilde{\cR}(D^2W)=0
 \qquad\text{in }B_1.
\]
Apply Lemma~\ref{perturb} with
\[
 G_0(M)=\kappa(n-1)Q(M)-1,
 \qquad
 H(M)=\widetilde{\cR}(M),
 \qquad
 \mathcal O=\operatorname{int}\mathcal C.
\]
The smallness required in that lemma follows from
\eqref{small} and
\eqref{rescale-small}.  Hence
\(
 W\in C^{2,\alpha}(B_{1/2})
\)
with a uniform estimate.  After scaling back, we obtain
\[
 \|D^2w\|_{L^\infty(B_{s/2}(y_0))}
 +s^\alpha
 [D^2w]_{C^{0,\alpha}(B_{s/2}(y_0))}
 \le C.
\]
A finite covering of \(B_{\kappa R/8}\) by such balls proves
\eqref{C2a}.  To obtain the global
seminorm in \eqref{scaled}, use the preceding local
estimate when \(|y-y'|<s/2\), centered at one of the two points.  If
\(|y-y'|\ge s/2\), then
\[
 \frac{|D^2w(y)-D^2w(y')|}{|y-y'|^\alpha}
 \le
 2\|D^2w\|_{L^\infty}\left(\frac2s\right)^\alpha.
\]
Since \(s=\kappa R/16\), these two cases give
\eqref{scaled}.
\end{proof}

\section{Exclusion of singular dual Hessians}
\label{exclude}

By Proposition~\ref{regularity}, the localized dual function $w$ is in
\(C^{2,\alpha}\), so its Hessian is defined pointwise.  Suppose for
the moment that
\[
 \det D^2w(y_0)=0
\]
at some point.  By \eqref{sing}, the kernel is
one-dimensional and the other \(n-1\) eigenvalues all equal
\(\kappa^{-1}\).  
We will show that \(D^2w\) remains singular near \(y_0\).  The proof
uses two changes of variables and a strong maximum principle. 

\begin{lemma}
\label{dini}
Let \(N\ge2\), let \(Z\in C(B_1)\) be nonnegative, and let
\(A\in L^\infty(B_1)\) satisfy
\begin{equation*}
 0<\lambda\le A\le\Lambda
 \quad\text{a.e. in }B_1.
\end{equation*}
Suppose
\begin{equation*}
 \partial_{11}(AZ)+\sum_{i=2}^{N}\partial_{ii}Z=0
 \qquad\text{in }\mathcal D'(B_1).
\end{equation*}
Assume that there are \(A_0>0\), \(r_0\in(0,1)\), and a nondecreasing
modulus \(\omega\) such that
\begin{equation*}
 \operatorname*{ess\,sup}_{B_r}|A-A_0|\le\omega(r)
 \quad(0<r<r_0),
 \qquad
 \int_0^{r_0}\frac{\omega(r)}r\,dr<\infty.
\end{equation*}
If \(Z(0)=0\), then \(Z\equiv0\) in a neighborhood of the origin.
\end{lemma}

\begin{proof}
Write \(X\) for the original variable and introduce \(y\) by
\[
 X_1=\sqrt{A_0}\,y_1,
 \qquad X'=y'.
\]
Define
\[
 \widetilde Z(y)=Z(\sqrt{A_0}y_1,y'),
 \qquad
 \widetilde A(y)=A(\sqrt{A_0}y_1,y').
\]
Since
\(\partial_{X_1}=A_0^{-1/2}\partial_{y_1}\), the original equation
becomes
\[
 \partial_{11}
 \left(\frac{\widetilde A}{A_0}\widetilde Z\right)
 +\sum_{i=2}^N\partial_{ii}\widetilde Z=0.
\]
After shrinking to a ball contained in the transformed ellipsoid and
relabeling \(y\) as \(x\) and \(\widetilde Z\) as \(Z\), this equation
becomes
\begin{equation}
 \Delta Z+\partial_{11}(bZ)=0
 \qquad\text{in }\mathcal D',
 \label{adjoint}
\end{equation}
where
\[
 b(x)=\frac{\widetilde A(x)}{A_0}-1.
\]
After changing \(b\) on a null set if necessary, the
essential-supremum bound gives
\[
|b(x)|\le \widetilde\omega(|x|)
\quad\text{for a.e. }x,
\]
where
\[
\widetilde\omega(r)=C\omega(Cr)
\]
is again a nondecreasing Dini modulus.
We henceforth rename \(\widetilde\omega\) as \(\omega\).
We shall prove that the spherical mean of \(Z\) vanishes.  
Let \(q=bZ\), and let \(d\sigma\) denote normalized surface measure on
\(\mathbb S^{N-1}\).  Define
\begin{equation*}
 H(s)
 =
 \int_{\mathbb S^{N-1}}Z(s\vartheta)\,d\sigma(\vartheta).
\end{equation*}
Then \(H\) is continuous and nonnegative, and
\[
 H(s)\longrightarrow Z(0)=0
 \qquad(s\downarrow0).
\]
For almost every \(s\), also define
\begin{equation*}
 \begin{split}
 Q(s)
 &=
 \int_{\mathbb S^{N-1}}
 \vartheta_1^2q(s\vartheta)\,d\sigma(\vartheta),\\
 P(s)
 &=
 \int_{\mathbb S^{N-1}}
 (1-\vartheta_1^2)q(s\vartheta)\,d\sigma(\vartheta).
 \end{split}
\end{equation*}

Let \(\Gamma\) be the fundamental solution of \(-\Delta\), normalized
by
\begin{equation*}
 \Gamma'(s)
 =
 -\frac1{|\mathbb S^{N-1}|s^{N-1}}.
\end{equation*}
This includes the logarithmic fundamental solution for \(N=2\).
Choose \(0<\varepsilon<r<r_0\) that are Lebesgue points of both
\(P\) and \(Q\).  Put
\[
 g(s)=\Gamma'(s)\mathbf 1_{(\varepsilon,r)}(s),
\]
let \(g_\delta\) be a one-dimensional smooth mollification, with
\(\delta<\varepsilon/4\), and define
\begin{equation*}
 \psi_\delta(s)
 =
 -\int_s^\infty g_\delta(\rho)\,d\rho.
\end{equation*}
For \(\delta\) small, \(\psi_\delta(|x|)\) is smooth, constant near
the origin, compactly supported inside the domain, and is therefore
an admissible test function in \eqref{adjoint}.

For a radial function \(\psi=\psi(s)\),
\begin{equation*}
 \partial_{11}\psi
 =
 \psi''\vartheta_1^2
 +\frac{\psi'}s(1-\vartheta_1^2),
 \qquad
 \Delta\psi
 =
 \psi''+\frac{N-1}{s}\psi'.
\end{equation*}
Testing \eqref{adjoint} with
\(\psi_\delta(|x|)\), using polar coordinates, and then letting
\(\delta\downarrow0\), we obtain
\begin{equation}
 \begin{split}
 0={}&
 H(r)-H(\varepsilon)
 +\int_{\{\varepsilon<|x|<r\}}
 q(x)\Gamma_{11}(x)\,dx\\
 &\quad-Q(\varepsilon)+Q(r).
 \end{split}
 \label{fund}
\end{equation}
Before passing to the
limit, the two integrals in the weak equation are
\begin{equation*}
 \begin{split}
 \int q\,\partial_{11}\psi_\delta
 &=
 |\mathbb S^{N-1}|
 \int_0^\infty
 \left(
 s^{N-1}g_\delta'(s)Q(s)
 +s^{N-2}g_\delta(s)P(s)
 \right)\,ds,\\
 \int Z\,\Delta\psi_\delta
 &=
 |\mathbb S^{N-1}|
 \int_0^\infty
 \left(
 s^{N-1}g_\delta'(s)
 +(N-1)s^{N-2}g_\delta(s)
 \right)H(s)\,ds.
 \end{split}
\end{equation*}
The interior terms converge to the integral involving
\(\Gamma_{11}\).  The jumps at \(\varepsilon\) and \(r\), together
with
\[
 |\mathbb S^{N-1}|s^{N-1}\Gamma'(s)=-1,
\]
give exactly the four boundary terms in
\eqref{fund}.  
Since \(\varepsilon\) and \(r\) are Lebesgue points of \(P\) and
\(Q\), the mollified radial integrals converge to the displayed
boundary terms.  Thus no trace of \(q\) on either sphere is required.

The classical Hessian of the fundamental solution satisfies
\[
 |\Gamma_{11}(x)|\le C_N|x|^{-N}.
\]
Using the pointwise bound for \(b\), \(Z\ge0\), and polar
coordinates, we find
\[
 \int_{B_r}|q|\,|\Gamma_{11}|\,dx
 \le
 C\int_0^r\omega(s)H(s)\,\frac{ds}{s}
 <\infty.
\]
At every good radius,
\[
 |Q(s)|\le\omega(s)H(s).
\]
Let \(\varepsilon\downarrow0\) through good radii in
\eqref{fund}.  The preceding bounds and
the convergence \(H(s)\to0\) give, for almost every sufficiently small
good \(r\),
\begin{equation*}
 H(r)
 +\int_{B_r}q(x)\Gamma_{11}(x)\,dx
 +Q(r)=0.
\end{equation*}
Therefore
\begin{equation*}
 H(r)
 \le
 C\int_0^r\omega(s)H(s)\,\frac{ds}{s}
 +\omega(r)H(r).
\end{equation*}
Since the Dini assumption implies \(\omega(r)\to0\), the last term can
be absorbed after reducing \(r_0\):
\[
 H(r)
 \le
 C\int_0^r\omega(s)H(s)\,\frac{ds}{s}
 \quad\text{for a.e. }r<r_0.
\]

Set
\[
 \mathcal H(r)
 =
 \int_0^r\omega(s)H(s)\,\frac{ds}{s}.
\]
It is absolutely continuous, \(\mathcal H(0)=0\), and
\[
 \mathcal H'(r)
 =
 \frac{\omega(r)}rH(r)
 \le
 C\frac{\omega(r)}r\mathcal H(r)
 \quad\text{for a.e. }r.
\]
For \(0<\delta<r\), by Gr\"onwall's inequality,
\[
 \mathcal H(r)
 \le
 \mathcal H(\delta)
 \exp\left(
 C\int_\delta^r\frac{\omega(s)}s\,ds
 \right).
\]
The exponential remains uniformly bounded as \(\delta\downarrow0\),
while \(\mathcal H(\delta)\to0\).  Thus
\(\mathcal H\equiv0\) near zero.  The preceding integral inequality
and continuity imply \(H\equiv0\).  Because \(Z\ge0\) is continuous, the
vanishing of all its spherical averages forces \(Z\equiv0\) in a
neighborhood of the origin.
\end{proof}

\begin{proposition}
\label{degeneracy}
Assume the hypotheses of Proposition~\ref{regularity}.  If
\(D^2w(y_0)\) is singular at some
\(y_0\in B_{\kappa R/8}\), then there are a neighborhood \(V\) of
\(y_0\) and a fixed orthogonal projection \(P\) onto an
\((n-1)\)-dimensional subspace such that
\[
D^2w=\kappa^{-1}P
\qquad\text{in }V.
\]
Consequently,
\begin{equation}
w(y)
a+b\cdot y+\frac1{2\kappa}|P(y-y_0)|^2,\quad y\in V.
\label{cylinder}
\end{equation}
\end{proposition}

\begin{proof}
After translating and rotating coordinates, assume that
\(y_0=0\), \(w(0)=0\), \(Dw(0)=0\), and
\[
D^2w(0)
=
\kappa^{-1}
\begin{pmatrix}
0&0\\
0&I_{n-1}
\end{pmatrix}.
\]
Since \(w\in C^{2,\alpha}\), every point is an Alexandrov point.
Hence \eqref{sing} shows that every singular Hessian has the form
\[
D^2w=\kappa^{-1}(I-\nu\otimes\nu).
\]
At points where \(D^2w>0\), we have
\[
\sigma_2\bigl((D^2w)^{-1}-\kappa I\bigr)=f(Dw).
\]
We now isolate the one-dimensional degeneracy.
Write
\[
 y=(t,\eta)\in\R\times\R^m,
 \qquad m=n-1.
\]
By the value of \(D^2w(0)\) above and continuity, after shrinking
the neighborhood the transverse block
\[
 B=D_{\eta\eta}^2w
\]
is uniformly positive definite.  Hence
\[
 (t,\eta)\longmapsto(t,x),
 \qquad
 x=D_\eta w(t,\eta),
\]
is a local \(C^{1,\alpha}\) diffeomorphism.  Let
\(\eta=\eta(t,x)\) be its inverse and define
\[
 h(t,x)
 =
 w(t,\eta(t,x))-x\cdot\eta(t,x)
 +\frac{\kappa}{2}|x|^2.
\]
Then \(h\in C^{2,\alpha}\).  Set
\begin{equation*}
 v=h_t,\qquad
 s=h_{tt},\qquad
 p=D_xh_t,\qquad
 T=-D_x^2h,\qquad
 \tau=\tr T.
\end{equation*}
If
\[
 a=w_{tt},
 \qquad
 b=D_\eta w_t,
\]
by differentiating these definitions, we obtain
\[
 h_t=w_t,\qquad
 D_xh=\kappa x-\eta,\qquad
 D_x\eta=B^{-1},\qquad
 \eta_t=-B^{-1}b.
\]
Consequently,
\begin{equation}
 p=B^{-1}b,\qquad
 s=a-b^TB^{-1}b,\qquad
 T=B^{-1}-\kappa I.
 \label{first}
\end{equation}
Equivalently,
\begin{equation}
 B=(T+\kappa I)^{-1},
 \qquad
 b=Bp,
 \qquad
 a=s+p^TBp,
 \label{first-inv}
\end{equation}
and
\begin{equation}
 D^2w
 =
 \begin{pmatrix}
 s+p^TBp&p^TB\\
 Bp&B
 \end{pmatrix}.
 \label{wblock}
\end{equation}
In particular \(s\ge0\), because it is the Schur complement of \(B\)
in the nonnegative matrix \(D^2w\).

At a point where \(s>0\), the matrix in
\eqref{wblock} is positive definite, and
\begin{equation}
 A=(D^2w)^{-1}-\kappa I
 =
 \begin{pmatrix}
 s^{-1}-\kappa&-s^{-1}p^T\\
 -s^{-1}p&T+s^{-1}p\otimes p
 \end{pmatrix}.
 \label{Ablock}
\end{equation}
Furthermore,
\begin{equation*}
 Dw=(v,x).
\end{equation*}
Using
\[
 \sigma_2
 \begin{pmatrix}
 a_0&b_0^T\\
 b_0&C_0
 \end{pmatrix}
 =
 a_0\tr C_0+\sigma_2(C_0)-|b_0|^2
\]
in \(\sigma_2(A)=f(Dw)\), we calculate
\begin{equation}
 (1+|p|^2-\kappa s)\tau
 =
 s\bigl(f(v,x)-\sigma_2(T)\bigr)
 +\kappa|p|^2+p^TTp.
 \label{first-eq}
\end{equation}
When \(m=1\), the convention is
\(\sigma_2(T)=\sigma_2(R)=0\) for the \(1\times1\) matrices occurring
in the two transformed equations below.

Equation \eqref{first-eq} remains valid at every
point where \(s=0\).  Indeed, by
\eqref{wblock} and \(B>0\), the equality \(s=0\) is
equivalent to singularity of \(D^2w\).  
Write the unit vector in \eqref{sing} as
\(e=(e_1,e').\)
Positivity of \(B\) implies \(e_1\ne0\).  Substituting this expression
into \eqref{first}, we obtain
\[
s=0,\qquad
p=-\frac{e'}{e_1},
\qquad
T=\kappa p\otimes p.
\]
Both sides of \eqref{first-eq} are then equal to
\(\kappa|p|^2(1+|p|^2)\).  Thus no density assertion about the set
\(\{s>0\}\) is being used.

Let \(d=1-\kappa s.\)
At the origin \(s=0\). After another shrinking we may assume \(d\ge\frac12.\)
Define
\begin{equation*}
 \theta=t-\kappa v(t,x),
 \qquad
 U(\theta,x)
 =
 h(t,x)-\frac{\kappa}{2}v(t,x)^2.
\end{equation*}
Since
\[
 \theta_t=d,\qquad D_x\theta=-\kappa p,
\]
the map \((t,x)\mapsto(\theta,x)\) is a local
\(C^{1,\alpha}\) diffeomorphism, whose inverse satisfies
\begin{equation*}
 t_\theta=\frac1d,
 \qquad
 D_xt=\frac{\kappa p}{d}.
\end{equation*}
Differentiating \(U\) at fixed \((\theta,x)\), we obtain
\[
 U_\theta=v,\qquad D_xU=D_xh.
\]
Moreover, we have
\(U\in C^{2,\alpha}\).
Let
\begin{equation*}
 Z=U_{\theta\theta},
 \qquad
 \zeta=D_xU_\theta,
 \qquad
 R=-D_x^2U.
\end{equation*}
We have
\begin{equation}
 Z=\frac{s}{d},
 \qquad
 \zeta=\frac{p}{d},
 \qquad
 R=T-\frac{\kappa}{d}p\otimes p.
 \label{second}
\end{equation}
In particular \(Z\ge0\).

If \(Z>0\), substituting
\eqref{second} into
\eqref{Ablock}, we obtain
\[
 A
 =
 \begin{pmatrix}
 Z^{-1}&-Z^{-1}\zeta^T\\
 -Z^{-1}\zeta&R+Z^{-1}\zeta\otimes\zeta
 \end{pmatrix}.
\]
Since \(A\ge0\), its Schur complement with respect to the upper-left
entry satisfies
\[
 R\ge0\quad\text{on }\{Z>0\}.
\]
If \(Z=0\), then \(s=0\); \eqref{second} give \(R=0\).  Therefore \(G\ge 0\) throughout the transformed neighborhood.

For any symmetric \(R\), vector \(\zeta\), and scalar \(\rho\),
\[
 \sigma_2(R+\rho\zeta\otimes\zeta)
 =
 \sigma_2(R)
 +\rho\bigl(|\zeta|^2\tr R-\zeta^TR\zeta\bigr).
\]
Thus on \(\{Z>0\}\), we have
\begin{equation}
 (1+|\zeta|^2)\tr R
 =
 Z\bigl(f(U_\theta,x)-\sigma_2(R)\bigr)
 +\zeta^TR\zeta.
 \label{second-eq}
\end{equation}
On \(\{Z=0\}\), by substituting
\eqref{sing}, we have \(R=0\), so
both sides vanish.  Thus
\eqref{second-eq} holds pointwise everywhere.
At the origin, \(Z(0)=0,\zeta(0)=0,\) and \(R(0)=0.\)

To apply Lemma~\ref{dini}, let
\begin{equation*}
 \hat f(\theta,x)
 =
 f(U_\theta(\theta,x),x)-\sigma_2(R(\theta,x)),
 \qquad
 f_0=f(0)>0.
\end{equation*}
Here \(f_0=f(0)\) by our normalization at the origin.  By the values
of \(Z\), \(\zeta\), and \(R\) at the origin and continuity, after
shrinking the neighborhood,
\[
 0<\mu\le\hat f\le M_0.
\]
On \(\{Z>0\}\), define
\begin{equation*}
 \beta=\frac{\tr R}{Z}.
\end{equation*}
Since \(R\ge0\),
\[
 0\le\zeta^TR\zeta\le|\zeta|^2\tr R,
\]
and by \eqref{second-eq},
\begin{equation}
 \frac{\hat f}{1+|\zeta|^2}
 \le\beta\le\hat f
 \qquad\text{on }\{Z>0\}.
 \label{beta-bounds}
\end{equation}
On \(\{Z=0\}\), define \(\beta=f_0\).  Since \(R=0\) on that set,
\(\beta\) is Borel measurable and
\begin{equation}
 \beta Z=\tr R=-\Delta_xU.
 \label{betaZ}
\end{equation}
The preceding bounds for \(\hat f\) and
\eqref{beta-bounds} imply, after shrinking once more,
\[
 0<\lambda\le\beta\le\Lambda.
\]

Because \(D^2U(0)=0\) and \(U\in C^{2,\alpha}\),
\[
 |Z(\theta,x)|+|\zeta(\theta,x)|+|R(\theta,x)|
 \le C|(\theta,x)|^\alpha.
\]
The map
\[
 (\theta,x)\longmapsto(U_\theta(\theta,x),x)
\]
is locally Lipschitz.  Hence the H\"older continuity of \(f\) and the
preceding estimate, together with
\[
 |\sigma_2(R(\theta,x))|
 \le C|(\theta,x)|^{2\alpha},
\]
give
\begin{equation*}
 |\mathfrak f(\theta,x)-f_0|
 \le C|(\theta,x)|^\alpha.
\end{equation*}
Using both sides of \eqref{beta-bounds}, and the definition of
\(\beta\) on \(\{Z=0\}\), we obtain
\begin{equation}
 \operatorname*{ess\,sup}_{B_r}|\beta-f_0|
 \le Cr^\alpha.
 \label{beta-mod}
\end{equation}
Indeed, on \(\{Z>0\}\),
\[
 |\beta-f_0|
 \le
 |\mathfrak f-f_0|+C|\zeta|^2,
\]
and on \(\{Z=0\}\) the left-hand side is zero.

Finally, \eqref{betaZ} and \(Z=U_{\theta\theta}\) imply, in the
sense of distributions,
\begin{equation}
 \partial_{\theta\theta}(\beta Z)+\Delta_xZ
 =
 -\partial_{\theta\theta}\Delta_xU
 +\Delta_xU_{\theta\theta}
 =0.
 \label{Zeq}
\end{equation}

Apply Lemma~\ref{dini} to
\eqref{Zeq}.
The uniform bounds for \(\beta\) and
\eqref{beta-mod} verify the ellipticity and the Dini
condition at the base point, with \(\omega(r)=Cr^\alpha\).  Since \(Z\ge0\) and
\(Z(0)=0\), the lemma implies
\(Z\equiv0\)
in a smaller neighborhood.

The identity \eqref{betaZ} and the vanishing of \(Z\) give
\(\tr R=0\).  Since \(R\ge0\),
\(
 R\equiv0.\)
On a product cylinder contained in the transformed neighborhood,
\[
 U_{\theta\theta}=0,
 \qquad
 D_x^2U=0.
\]
The first equality implies
\[
 U(\theta,x)=\theta g(x)+k(x).
\]
Substituting into the second equality for every \(\theta\) in an
interval, we obtain
\[
 D_x^2g=D_x^2k=0.
\]
Thus \(g\) and \(k\) are affine, and
\(\zeta=D_xU_\theta=Dg\) is constant.  Since \(\zeta(0)=0\),
\(\zeta\equiv0.\)

The identities \eqref{second} invert to
\begin{equation*}
 s=\frac{Z}{1+\kappa Z},
 \qquad
 d=\frac1{1+\kappa Z},
 \qquad
 p=d\zeta,
 \qquad
 T=R+\frac{\kappa}{d}p\otimes p.
\end{equation*}
The three vanishing conclusions above therefore imply
\[
 s\equiv0,\qquad p\equiv0,\qquad T\equiv0.
\]
Using
\eqref{first-inv}--\eqref{wblock},
we conclude that
\[
 D^2w
 =
 \begin{pmatrix}
 0&0\\
 0&\kappa^{-1}I_{n-1}
 \end{pmatrix}
\]
throughout a neighborhood in the original \((t,\eta)\) variables.
The two changes of variables above are local diffeomorphisms, so this
is an open neighborhood of \(y_0\).  Integrating the constant Hessian
proves \eqref{cylinder}.
\end{proof}

\begin{proposition}
\label{positive}
Under the assumption of
Proposition~\ref{regularity}, we have \(D^2w(y)>0\)
for every \(y\in B_{\kappa R/8}\).
\end{proposition}

\begin{proof}
Assume for contradiction that \(D^2w(y_0)\) is singular at some
\(y_0\in B_{\kappa R/8}\).  Proposition~\ref{degeneracy} gives a
neighborhood \(V\) of \(y_0\), a fixed orthogonal projection \(P\) of
rank \(n-1\), and constants \(a,b\) such that \eqref{cylinder} holds.
After shrinking \(V\), assume that \(V\Subset B_{\kappa R/4}.\)
By Proposition~\ref{locdual}, for every \(y\in V\), the supremum in
\eqref{dualdef} has a unique interior maximizer
\[
x(y)=Dw(y)\in B_R.
\]
Fenchel equality gives
\[
w(y)+\Phi(x(y))=x(y)\cdot y,
\]
and hence
\[
y\in\partial\Phi(x(y)).
\]

We record explicitly the corresponding subgradient of \(u\).  Fix
\(y\in V\), put \(x=x(y)\), and define
\[
 p=y-\kappa x.
\]
For \(z\in B_R\), set \(d=z-x\).  If \(t>0\) is small enough, then
\(x+td\in B_R\), and by the subgradient inequality for \(\Phi\),
\[
 u(x+td)-u(x)
 \ge
 t(y-\kappa x)\cdot d-\frac{\kappa}{2}t^2|d|^2.
\]
After division by \(t\) and passage to the limit \(t\downarrow0\),
\[
 u'(x;d)\ge p\cdot d.
\]
The secant slopes of a convex function are nondecreasing, so
\[
 u(z)-u(x)\ge u'(x;d)\ge p\cdot(z-x).
\]
Consequently,
\begin{equation}
 y-\kappa x(y)\in\partial u(x(y))
 \quad\text{relative to }B_R.
 \label{subgrad}
\end{equation}
Let \(e\) be a unit vector spanning \(\ker P\), so
\(P=I-e\otimes e\).  Choose \(\bar y\in V\), set
\[
 c_0=e\cdot\bar y,
\]
and consider the relatively open slice
\[
 S
 =
 V\cap\{y:e\cdot y=c_0\}.
\]
After absorbing the fixed vector \(Py_0\) into the affine term, the
gradient of \eqref{cylinder} can be written
\[
 x(y)=Dw(y)=\widetilde b+\kappa^{-1}Py.
\]
The restriction of \(P\) to
\(\{e\cdot y=c_0\}\) is an affine isomorphism onto \(e^\perp\).
Therefore
\[
 \Sigma:=Dw(S)
\]
is a relatively open \((n-1)\)-dimensional subset of the affine
hyperplane \(\widetilde b+e^\perp\).  The fact that the maximizers are
interior implies \(\Sigma\subset B_R\).

For \(y\in S\),
\[
 y-\kappa x(y)
 =
 y-\kappa\widetilde b-Py
 =
 c_0e-\kappa\widetilde b=:p_*.
\]
By \eqref{subgrad}, the same vector \(p_*\) belongs
to \(\partial u(x)\) for every \(x\in\Sigma\).  If
\(x_1,x_2\in\Sigma\), the two subgradient inequalities give
\[
 \begin{split}
 u(x_2)&\ge u(x_1)+p_*\cdot(x_2-x_1),\\
 u(x_1)&\ge u(x_2)+p_*\cdot(x_1-x_2),
 \end{split}
\]
and hence equality holds.  Fix \(x_*\in\Sigma\) and put
\[
 L(x)=u(x_*)+p_*\cdot(x-x_*).
\]
Then \(L\) supports \(u\) throughout \(B_R\), and
\begin{equation}
 \Sigma\subset\{x\in B_R:u(x)=L(x)\}.
 \label{patch}
\end{equation}
Thus the contact set has dimension at least \(n-1\).

Set
\[
 \widetilde u=m^{-1/2}u,
 \qquad
 \widetilde L=m^{-1/2}L.
\]
By the homogeneity of degree two of \(\sigma_2\),
\begin{equation*}
 \sigma_2(D^2\widetilde u)\ge1
 \qquad\text{in }B_R
\end{equation*}
in the viscosity sense.  Indeed, if a \(2\)-convex function
\(\varphi\in C^2\) touches \(\widetilde u\) from above, then
\(\sqrt m\,\varphi\) touches \(u\) from above and
\[
 m\,\sigma_2(D^2\varphi)
 =
 \sigma_2(D^2(\sqrt m\,\varphi))
 \ge f\ge m.
\]
By Mooney's strict \(2\)-convexity theorem
\cite[Theorem~1.1]{Mooney}, applied to
\((\widetilde u,\widetilde L)\),
we obtain
\[
 \dim\{x\in B_R:\widetilde u(x)=\widetilde L(x)\}\le n-2.
\]
This contradicts \eqref{patch}.  Hence
the assumed singular point cannot exist.
\end{proof}

\section{Completion of the qualitative proof}

We have proved that the Hessian of the localized dual is positive
definite.  We now invert the transform and finish the qualitative
part of Theorem~\ref{main}.

\begin{lemma}
\label{inverse}
Let \(g:\mathbb R^n\to(-\infty,+\infty]\) be proper, closed, and
convex, and let \(w=g^*\).  Suppose that
\(w\in C^{2,\alpha}\) near \(y_0\), and set
\(x_0=Dw(y_0).\)
If \(D^2w(y_0)>0\), then there are neighborhoods \(V\ni y_0\) and
\(U\ni x_0\) such that
\(Dw:V\rightarrow U\)
is a \(C^{1,\alpha}\) diffeomorphism.  If
\(\Psi=(Dw)^{-1}\), then
\begin{equation}
 g(x)=x\cdot\Psi(x)-w(\Psi(x)),
 \qquad
 Dg(x)=\Psi(x),
 \label{local-conj}
\end{equation}
and
\begin{equation}
 D^2g(x)
 =
 \bigl[D^2w(\Psi(x))\bigr]^{-1}.
 \label{inv-hess}
\end{equation}
In particular, \(g\in C^{2,\alpha}(U)\).

Suppose, in addition, that \(u\in C(\overline{B_R})\) is convex and
\(\kappa>0\), and set
\[
 g=\Phi+I_{\overline{B_R}},
 \qquad
 \Phi=u+\frac{\kappa}{2}|x|^2,
 \qquad
 w=g^*,
\]
where \(I_{\overline{B_R}}\) is zero on \(\overline{B_R}\) and
\(+\infty\) outside it.  If \(x_0=Dw(y_0)\in B_R\), then
\(u\in C^{2,\alpha}\) near
\(x_0\), with
\begin{equation*}
 D^2u(x)
 =
\bigl[D^2w(\Psi(x))\bigr]^{-1}-\kappa I.
\end{equation*}
\end{lemma}

\begin{proof}
Continuity and strict positivity of \(D^2w(y_0)\) give, after
shrinking the neighborhood,
\[
 D^2w\ge\lambda_0 I
\]
for some \(\lambda_0>0\).  The \(C^{1,\alpha}\) inverse function
theorem applied to \(Dw\) provides neighborhoods \(V,U\) and a
\(C^{1,\alpha}\) inverse \(\Psi\), with
\begin{equation}
 D\Psi(x)
 =
 \bigl[D^2w(\Psi(x))\bigr]^{-1}.
 \label{inv-map}
\end{equation}

For \(x\in U\), put \(y=\Psi(x)\).  Then
\(x=Dw(y)\in\partial w(y)\).  Since \(w=g^*\), by Fenchel duality,
\[
 g(x)+w(y)=x\cdot y.
\]
This proves the first identity in
\eqref{local-conj}.  The function on its
right-hand side is \(C^{1,\alpha}\).  Differentiating and using
\(Dw(\Psi(x))=x\), we obtain
\[
 \begin{split}
 Dg(x)
 &=
 \Psi(x)+(D\Psi(x))^Tx
 -(D\Psi(x))^TDw(\Psi(x))\\
 &=\Psi(x).
 \end{split}
\]
We differentiate once more and use
\eqref{inv-map} to obtain
\eqref{inv-hess}.  Because matrix inversion is smooth
on the positive definite cone bounded away from its boundary, the
right-hand side is \(C^{0,\alpha}\), with no loss of exponent.

For the localized statement, the extension \(g\) is proper, closed,
and convex, while
\(w=g^*\) and \(w^*=g\) by the Fenchel--Moreau theorem.
By the general part just proved,
\(g\in C^{2,\alpha}\) near \(x_0\).  Since \(x_0\in B_R\), that
neighborhood may be reduced so that \(g=\Phi\) there.  Subtracting
\(\kappa|x|^2/2\) and using
\eqref{inv-hess} proves the result.

\end{proof}

\begin{proof}[Proof of Theorem~\ref{main}]
Fix an arbitrary point \(x_0\in\Omega\), and translate \(x_0\) to the
origin.  Since \(u\) is
convex, it has a supporting affine function at the origin.  After
subtracting this function, we have
\begin{equation}
 u(0)=0,\qquad 0\in\partial u(0),\qquad u\ge0
 \label{normalize}
\end{equation}
on a neighborhood of the origin.

First choose \(R_0>0\) with
\(\overline{B_{2R_0}}\subset\Omega\).  Since \(f(0)>0\) and \(f\) is
continuous, after reducing \(R_0\) we may fix
\[
 m=\frac12f(0),\qquad M=2f(0),
 \qquad
 m\le f\le M
 \quad\text{on }B_{2R_0}.
\]
The scale-invariant oscillation tends to zero:
\[
 R^\alpha[f]_{C^{0,\alpha}(B_{2R})}\longrightarrow0,
 \qquad R\downarrow0.
\]
We may therefore choose \(0<R\le R_0\) so that the smallness hypothesis
\eqref{small} of
Proposition~\ref{regularity} holds.

Define \(c,\kappa,\Phi\), and \(w\) by
\eqref{kappa} and
\eqref{dualdef}.  Condition
\eqref{normalize} makes \(x=0\) the unique maximizer in
the definition of \(w(0)\). Hence
\(Dw(0)=0.\)
By Propositions~\ref{regularity}
and~\ref{positive},
\(
 w\in C^{2,\alpha}(B_{\kappa R/8}),
 D^2w(0)>0.
\)
The maximizer \(Dw(0)=0\) lies in \(B_R\).  Applying the localized
form of Lemma~\ref{inverse} at \(y_0=0\), we obtain
\(u\in C^{2,\alpha}\)
on a neighborhood of the original point \(x_0\).  Since \(x_0\) was
arbitrary, \(u\in C_{\mathrm{loc}}^{2,\alpha}(\Omega)\).

\end{proof}

\section{Quantitative analysis}

Sections~2--5 exclude a singular dual Hessian for each fixed solution.
We now obtain a lower bound at the center of each localized dual
transform which is uniform over every family satisfying fixed data
bounds.
Throughout this section,
\(\|f\|_{C^{0,\alpha}}\) denotes the full norm, including
\(\|f\|_{L^\infty}\).

\begin{proposition}
\label{lower-bound}
Fix
\[
 n\ge2,\qquad 0<\alpha<1,\qquad
 K<\infty,\qquad 0<\lambda\le L<\infty.
\]
Let
\[
 \varepsilon_*=\varepsilon_*(n,\alpha,\lambda,L)
\]
be the smallness constant in
Proposition~\ref{regularity}, and set
\begin{equation*}
 R_*=
 \min\left\{
 \frac14,
 \left(\frac{\varepsilon_*}{2L}\right)^{1/\alpha}
 \right\},
 \qquad
 \kappa_-=
 \left(\frac{\lambda}{\binom n2}\right)^{1/2},
 \qquad
 \rho_*=\frac{\kappa_-R_*}{32}.
\end{equation*}
Then there is
\[
 \mu=\mu(n,\alpha,K,\lambda,L)>0
\]
with the following property.

Let \(u\in C(B_2)\) be convex and solve
\[
 \sigma_2(D^2u)=f
 \qquad\text{in }B_2
\]
in the admissible viscosity sense, where
\begin{equation}
 \|u\|_{L^\infty(B_2)}\le K,\qquad
 f\ge\lambda,\qquad
 \|f\|_{C^{0,\alpha}(B_2)}\le L.
 \label{data}
\end{equation}
For \(x_0\in B_{1/2}\) and \(p\in\partial u(x_0)\), define
\begin{equation}
 U_{x_0,p}(z)
 =
 u(x_0+z)-u(x_0)-p\cdot z,
 \qquad
 F_{x_0}(z)=f(x_0+z),
 \label{normalized-u}
\end{equation}
\begin{equation*}
 c_{x_0}=\min_{\overline{B_{R_*}}}F_{x_0},
 \qquad
 \binom n2\kappa_{x_0}^2=c_{x_0},
\end{equation*}
and
\begin{equation}
 w_{x_0,p}(y)
 =
 \sup_{z\in\overline{B_{R_*}}}
 \left\{
 z\cdot y-U_{x_0,p}(z)
 -\frac{\kappa_{x_0}}2|z|^2
 \right\}.
 \label{family-w}
\end{equation}
Then
\begin{equation}
 D^2w_{x_0,p}(0)\ge\mu I,
 \label{mu}
\end{equation}
uniformly over all solutions, base points, and supporting vectors in
the displayed class.
\end{proposition}

\begin{proof}
The conversion of qualitative strictness into a uniform modulus is a
compactness argument in the same broad spirit as
\cite{McGonagleSongYuan,Mooney}.  We first fix the localization
uniformly and obtain compactness of the
original functions.
For every unit vector \(e\), the points \(x_0\pm e\) lie in \(B_2\).
The two supporting inequalities give
\[
 |p\cdot e|\le2K,
\]
and hence
\begin{equation}
 |p|\le2K.
 \label{slopes}
\end{equation}
Thus \(U_{x_0,p}(0)=0\), \(0\in\partial U_{x_0,p}(0)\),
\(U_{x_0,p}\ge0\), and
\begin{equation*}
 0\le U_{x_0,p}(z)
 \le2K+2K|z|
 \le\frac92K
 \qquad(|z|<5/4).
\end{equation*}
Since \(R_*\le1/4\),
\[
 B_{2R_*}(x_0)\subset B_1.
\]
Moreover,
\begin{equation*}
 \kappa_-\le\kappa_{x_0}\le
 \kappa_+:=
 \left(\frac{L}{\binom n2}\right)^{1/2},
 \qquad
 R_*^\alpha[F_{x_0}]_{C^{0,\alpha}(B_{2R_*})}
 \le\frac{\varepsilon_*}{2}.
\end{equation*}
Proposition~\ref{regularity} therefore applies uniformly.
In particular,
\[
 \overline{B_{2\rho_*}}
 \Subset B_{\kappa_{x_0}R_*/8}
\]
with a uniform positive margin.

We now pass to a compactness argument for the translated functions
and right-hand sides.
If \eqref{mu} were false, there would be
\[
 (u_j,f_j,x_j,p_j),
\qquad
 x_j\in B_{1/2},\quad
 p_j\in\partial u_j(x_j),
\]
satisfying \eqref{data}, such that
\begin{equation}
 \lambda_{\min}\bigl(D^2w_j(0)\bigr)\longrightarrow0.
 \label{collapse}
\end{equation}
Convexity and the amplitude bound make \(u_j\) uniformly Lipschitz on
\(\overline{B_{3/2}}\).  Indeed, for
\(q\in\partial u_j(x)\), \(x\in B_{3/2}\), evaluate the supporting
inequality at \(x+tq/|q|\) and let \(t\uparrow1/2\) to get
\[
 |q|\le4K.
\]
The functions \(f_j\) have a common H\"older modulus on the same
compact ball.  Using also \eqref{slopes}, we pass to a
subsequence such that
\begin{equation*}
 \begin{split}
 u_j&\longrightarrow u_\infty
 \quad\text{uniformly on }\overline{B_{3/2}},\\
 f_j&\longrightarrow f_\infty
 \quad\text{uniformly on }\overline{B_{3/2}},\\
 x_j&\longrightarrow x_\infty,\qquad
 p_j\longrightarrow p_\infty.
 \end{split}
\end{equation*}
The supporting inequalities pass to the limit, so
\[
 p_\infty\in\partial u_\infty(x_\infty).
\]

On \(B_1\), put
\[
 U_j(z)=u_j(x_j+z)-u_j(x_j)-p_j\cdot z,
 \qquad
 F_j(z)=f_j(x_j+z).
\]
Then, uniformly on \(\overline{B_1}\),
\begin{equation*}
 U_j\longrightarrow U_\infty,\qquad
 F_j\longrightarrow F_\infty,
\end{equation*}
where
\[
 U_\infty(z)
 =
 u_\infty(x_\infty+z)-u_\infty(x_\infty)-p_\infty\cdot z,
 \qquad
 F_\infty(z)=f_\infty(x_\infty+z).
\]
The limit is convex and satisfies
\[
 U_\infty(0)=0,\qquad
 0\in\partial U_\infty(0),\qquad
 U_\infty\ge0,
\]
while
\begin{equation*}
 F_\infty\ge\lambda,\qquad
 \|F_\infty\|_{L^\infty(B_1)}\le L,\qquad
 [F_\infty]_{C^{0,\alpha}(B_1)}\le L.
\end{equation*}

By viscosity stability,
\begin{equation*}
 \sigma_2(D^2U_\infty)=F_\infty
 \qquad\text{in }B_1.
\end{equation*}
To verify this within the admissible convention, make an upper contact
strict by adding a small positive quadratic term and a lower contact
strict by subtracting one.  Uniform convergence produces nearby
contact points for \(U_j\), and the openness of \(\Gamma_2\) preserves
the lower test's admissibility.  If lower tests on
\(\partial\Gamma_2\) are included in the convention, the desired
inequality is automatic there because
\(\sigma_2(D^2\phi)=0\le F_\infty\).  Passing first \(j\to\infty\) and
then the quadratic perturbation to zero proves both inequalities.

The parameters in the quadratic perturbations also converge.
Let
\[
 c_j=\min_{\overline{B_{R_*}}}F_j,\qquad
 \binom n2\kappa_j^2=c_j.
\]
By uniform convergence on the fixed compact ball,
\begin{equation*}
 c_j\longrightarrow
 c_\infty:=\min_{\overline{B_{R_*}}}F_\infty,
 \qquad
 \kappa_j\longrightarrow
 \kappa_\infty:=
 \left(\frac{c_\infty}{\binom n2}\right)^{1/2}.
\end{equation*}
In particular,
\[
 \kappa_-\le\kappa_\infty\le\kappa_+,
 \qquad
 c_\infty=\binom n2\kappa_\infty^2.
\]

We next pass to the localized Legendre transforms.
Define, on \(\overline{B_{R_*}}\),
\[
 \Phi_j(z)=U_j(z)+\frac{\kappa_j}{2}|z|^2,
 \qquad
 \Phi_\infty(z)=U_\infty(z)+\frac{\kappa_\infty}{2}|z|^2,
\]
and let \(w_j,w_\infty\) be their restricted conjugates.  Then
\begin{equation}
 \sup_{y\in\R^n}|w_j(y)-w_\infty(y)|
 \le
 \|U_j-U_\infty\|_{L^\infty(B_{R_*})}
 +\frac{R_*^2}{2}|\kappa_j-\kappa_\infty|
 \longrightarrow0.
 \label{conj-stab}
\end{equation}
Thus \(w_\infty\) is exactly the localized conjugate of
\((U_\infty,\kappa_\infty)\).

If \(\xi_j(y)\) and \(\xi_\infty(y)\) are the unique maximizers, the
uniform \(\kappa_-\)-strong convexity of
\(\Phi_j-y\cdot z\) implies, with the right-hand side of
\eqref{conj-stab} denoted by \(\delta_j\),
\begin{equation}
 \frac{\kappa_-}{2}
 |\xi_j(y)-\xi_\infty(y)|^2
 \le2\delta_j.
 \label{max-stab}
\end{equation}
For \(y\in\overline{B_{2\rho_*}}\), the boundary comparison
\[
 z\cdot y-U_j(z)-\frac{\kappa_j}{2}|z|^2
 \le R_*|y|-\frac{\kappa_jR_*^2}{2}<0
 \qquad(|z|=R_*)
\]
keeps all maximizers in the interior.  Hence
\[
 Dw_j=\xi_j\longrightarrow\xi_\infty=Dw_\infty
\]
uniformly on the common dual ball.

We next prove strong convergence of the dual Hessians.
By the scaled estimate \eqref{scaled}, with
\(m=\lambda\), \(M=L\),
\begin{equation*}
 \|D^2w_j\|_{L^\infty(B_{2\rho_*})}
 +[D^2w_j]_{C^{0,\alpha}(B_{2\rho_*})}
 \le C_{\mathrm d},
\end{equation*}
where
\[
 C_{\mathrm d}
 =
 C_{\mathrm d}(n,\alpha,\lambda,L)<\infty.
\]
By the Arzel\`a--Ascoli theorem, after passing to a subsequence,
\[
 D^2w_j\longrightarrow H
 \quad\text{uniformly on }\overline{B_{2\rho_*}}
\]
for a continuous matrix field \(H\).
Since \(\Phi_j\) has its unique minimum at zero,
\[
 w_j(0)=0,\qquad Dw_j(0)=0.
\]
For every segment
\([y,y+h]\Subset B_{2\rho_*}\),
\[
 Dw_j(y+h)-Dw_j(y)
 =
 \int_0^1D^2w_j(y+th)h\,dt.
\]
Pass to the limit using
\eqref{max-stab}.  This proves
\[
 H=D^2w_\infty
\]
and hence
\begin{equation}
 D^2w_j\longrightarrow D^2w_\infty
 \quad\text{uniformly on }\overline{B_{2\rho_*}}.
 \label{hess-conv}
\end{equation}
The common H\"older estimate also implies
\(w_\infty\in C^{2,\alpha}(B_{2\rho_*})\).
Equations \eqref{collapse} and
\eqref{hess-conv} now imply
\begin{equation}
 \lambda_{\min}\bigl(D^2w_\infty(0)\bigr)=0.
 \label{limit-sing}
\end{equation}

We now apply the qualitative result to the limit.
The limiting pair \((U_\infty,F_\infty)\) satisfies all the hypotheses
of Proposition~\ref{regularity} on \(B_{2R_*}\), with Lewy
parameter \(\kappa_\infty\).  Consequently
Proposition~\ref{positive} applies to its localized
conjugate \(w_\infty\).  Since
\(0\in B_{\kappa_\infty R_*/8}\), that proposition gives
\(D^2w_\infty(0)>0\), contradicting
\eqref{limit-sing}.

It follows that the infimum of
\(\lambda_{\min}(D^2w_{x_0,p}(0))\) over the class in the statement is
positive.  Taking this infimum, or one half of it, as \(\mu\) proves
\eqref{mu}.  Every compactness bound and every
fixed radius used above depends only on
\((n,\alpha,K,\lambda,L)\).
\end{proof}

The amplitude bound in
Proposition~\ref{lower-bound} cannot be omitted.
Let
\[
 A_t=\operatorname{diag}(t,t^{-1},0,\ldots,0),
 \qquad
 u_t(x)=\frac12x^TA_tx.
\]
Then \(u_t\) is convex and
\[
 \sigma_2(D^2u_t)=1,
\]
but, for the local Legendre--Lewy dual at the origin,
\[
 D^2w_t(0)=(A_t+\kappa I)^{-1},
 \qquad
 \lambda_{\min}(D^2w_t(0))
 =(t+\kappa)^{-1}\longrightarrow0.
\]
At the same time,
\(\|u_t\|_{L^\infty(B_2)}\) grows like \(t\).

\begin{proof}[Proof of the estimate in Theorem~\ref{main}]
Set
\[
 K=\|u\|_{L^\infty(B_2)},\qquad
 \lambda=\|f^{-1}\|_{L^\infty(B_2)}^{-1},\qquad
 L=\|f\|_{C^{0,\alpha}(B_2)}.
\]
Then \(0<\lambda\le L\).
Let \(\mu\) be given by
Proposition~\ref{lower-bound}.  The qualitative part
of the theorem, already proved in Section~5, implies that
\(u\in C^{2,\alpha}(B_2)\).

Fix \(x_0\in B_{1/2}\) and set \(p=Du(x_0)\).  For the functions in
\eqref{normalized-u}--\eqref{family-w}, the
unique maximizer defining \(w_{x_0,p}(0)\) is \(z=0\).  Thus
\(Dw_{x_0,p}(0)=0\).  By the second-order duality relation
\eqref{hessrel} and
\eqref{mu},
\[
 0\le D^2u(x_0)
 =
 \bigl[D^2w_{x_0,p}(0)\bigr]^{-1}
 -\kappa_{x_0}I
 \le\mu^{-1}I.
\]
Since \(x_0\) was arbitrary,
\begin{equation}
 \|D^2u\|_{L^\infty(B_{1/2})}\le\mu^{-1}.
 \label{u-hess}
\end{equation}

It remains to estimate the H\"older seminorm of \(D^2u\).  Set
\[
 \Lambda_0=\mu^{-1},
 \qquad
 F(A)=\sqrt{\sigma_2(A)},
 \qquad
 g=\sqrt f.
\]
In particular, \(\|g\|_{C^{0,\alpha}(B_{1/2})}\) is bounded in terms
of \(\lambda\) and \(L\).
Consider the compact set
\[
 \mathcal H
 =
 \{A\in\Sym(n):0\le A\le\Lambda_0I,
 \ \lambda\le\sigma_2(A)\le L\}.
\]
It is contained in \(\Gamma_2\), and \(F\) is smooth and concave on
a fixed neighborhood of \(\mathcal H\); see, for example,
\cite[Section~1]{CNS}.  To record uniform
ellipticity, diagonalize \(A\in\mathcal H\), let
\(a_1,\ldots,a_n\) be its eigenvalues, and put
\(s_i=\sum_{j\ne i}a_j\).  Then
\[
 F_i(A)=\frac{s_i}{2\sqrt{\sigma_2(A)}}.
\]
Moreover,
\[
 \sigma_2(A)
 =a_is_i+\sigma_2(a|i)
 \le\Lambda_0s_i+\frac12s_i^2
 \le\frac{n+1}{2}\Lambda_0s_i.
\]
Consequently,
\begin{equation}
 \frac{\lambda}{(n+1)\Lambda_0\sqrt L}
 \le F_i(A)
 \le\frac{(n-1)\Lambda_0}{2\sqrt\lambda}.
 \label{ellipticity}
\end{equation}

Define a concave extension of \(F\) by
\[
 \overline F(M)
 =
 \inf_{A\in\mathcal H}
 \{F(A)+DF(A):(M-A)\}.
\]
By \eqref{ellipticity}, \(\overline F\) is
uniformly elliptic on \(\Sym(n)\), with constants depending only on
\(n,\lambda,L\), and \(\mu\).  Concavity of \(F\) implies that
\(\overline F=F\) on \(\mathcal H\).  Hence
\[
 \overline F(D^2u)=g
 \qquad\text{in }B_{1/2}.
\]
By the Evans--Krylov estimate \cite{Evans,Krylov}, \(u\) has a uniform
\(C^{2,\gamma_0}\) bound on a smaller ball, where
\(\gamma_0>0\) depends only on the ellipticity constants.

We now use the standard nonlinear Schauder bootstrap.  If
\(x_0\in B_{3/8}\), \(A_0=D^2u(x_0)\), and \(L_0=DF(A_0)\), then
smoothness of \(F\) near \(\mathcal H\) gives
\[
 \bigl|
 L_0:(D^2u-A_0)-\bigl(F(D^2u)-F(A_0)\bigr)
 \bigr|
 \le C|D^2u-A_0|^2.
\]
The constant-coefficient Schauder estimate improves a
\(C^{0,\gamma}\) modulus of \(D^2u\) to a
\(C^{0,\min\{\alpha,2\gamma\}}\) modulus.  Iteration on nested balls,
with the endpoint Campanato estimate, reaches the exponent \(\alpha\).
The standard interior estimate for concave fully nonlinear equations
\cite[Chapters~6 and~8]{CaffarelliCabre} therefore gives
\[
 \|D^2u\|_{C^{0,\alpha}(B_{1/4})}
 \le C(n,\alpha,K,\lambda,L).
\]
Convexity and the amplitude bound, tested at \(x\pm e\) for unit
vectors \(e\), also give
\[
 \|Du\|_{L^\infty(B_{1/4})}\le2K.
\]
Together with \(\|u\|_{L^\infty(B_{1/4})}\le K\) and
\eqref{u-hess}, the last two estimates prove
\eqref{main-est}.
\end{proof}

\section*{Acknowledgement}
This work initiated from an effective discussion with Ling Wang and Yang Zhou, at the Westlake Young‑Scholars Workshop on PDE, 2026. We sincerely thank them for the helpful thought.


\begin{thebibliography}{99}

\bibitem{CaffarelliMA}
L.~A. Caffarelli,
\emph{Interior \(W^{2,p}\) estimates for solutions of the
Monge--Amp\`ere equation},
Ann. of Math. (2) \textbf{131} (1990), no.~1, 135--150.

\bibitem{CaffarelliCabre}
L.~A. Caffarelli and X.~Cabr\'e,
\emph{Fully Nonlinear Elliptic Equations},
American Mathematical Society Colloquium Publications, vol.~43,
American Mathematical Society, 1995.

\bibitem{CaffarelliGuanMa}
L.~A. Caffarelli, P.~Guan, and X.-N. Ma,
\emph{A constant rank theorem for solutions of fully nonlinear
elliptic equations},
Comm. Pure Appl. Math. \textbf{60} (2007), no.~12, 1769--1791.

\bibitem{CNS}
L.~Caffarelli, L.~Nirenberg, and J.~Spruck,
\emph{The Dirichlet problem for nonlinear second-order elliptic
equations, III: Functions of the eigenvalues of the Hessian},
Acta Math. \textbf{155} (1985), 261--301.

\bibitem{ChangYuan}
S. Y. A. Chang and Y. Yuan,
\emph{A Liouville problem for the sigma-2 equation},
Discrete Contin. Dyn. Syst. \textbf{28} (2010), no.~2, 659--664.

\bibitem{ChenJianTuZhou}
R.~Chen, H.~Jian, X.~Tu, and X.~Zhou,
\emph{Regularity for convex viscosity solutions of
\(\sigma_2\) equation},
preprint, 2026,
\href{https://arxiv.org/abs/2605.30823}{arXiv:2605.30823}.

\bibitem{ChouWang}
K.-S. Chou and X.-J. Wang,
\emph{A variational theory of the Hessian equation},
Comm. Pure Appl. Math. \textbf{54} (2001), no.~9, 1029--1064.

\bibitem{Evans}
L.~C. Evans,
\emph{Classical solutions of fully nonlinear, convex, second-order
elliptic equations},
Comm. Pure Appl. Math. \textbf{35} (1982), no.~3, 333--363.

\bibitem{Fan}
Z.~Fan,
\emph{A generalization of Savin's small perturbation theorem for
fully nonlinear elliptic equations and applications},
Calc. Var. Partial Differential Equations \textbf{65} (2026),
Art.~182,
\href{https://doi.org/10.1007/s00526-026-03349-7}
{doi:10.1007/s00526-026-03349-7}.

\bibitem{FanVariable}
Z.~Fan,
\emph{Hessian estimates for the sigma-2 equation with variable
right-hand side terms in dimension 4},
Adv. Math. \textbf{494} (2026), Art.~110953.

\bibitem{GuanPhong}
P.~Guan and D.~H. Phong,
\emph{Partial Legendre transforms of non-linear equations},
Proc. Amer. Math. Soc. \textbf{140} (2012), no.~11, 3831--3842.

\bibitem{GuanQiu}
P.~Guan and G.~Qiu,
\emph{Interior \(C^2\) regularity of convex solutions to prescribing
scalar curvature equations},
Duke Math. J. \textbf{168} (2019), no.~9, 1641--1663.

\bibitem{JianWang}
H.-Y. Jian and X.-J. Wang,
\emph{Continuity estimates for the Monge--Amp\`ere equation},
SIAM J. Math. Anal. \textbf{39} (2007), no.~2, 608--626.

\bibitem{Krylov}
N.~V. Krylov,
\emph{Boundedly inhomogeneous elliptic and parabolic equations},
Math. USSR Izv. \textbf{20} (1983), no.~3, 459--492.

\bibitem{LiWu}
K.~Wu and Z.~Li, 
\emph{Interior Hessian estimates for the quadratic Hessian equation},arXiv:2608.23233.

\bibitem{LianZhang}
Y.~Lian and K.~Zhang,
\emph{Pointwise regularity for locally uniformly elliptic equations
and applications},
preprint, 2024,
\href{https://arxiv.org/abs/2405.07199}{arXiv:2405.07199}.

\bibitem{McGonagleSongYuan}
M.~McGonagle, C.~Song, and Y.~Yuan,
\emph{Hessian estimates for convex solutions to quadratic Hessian
equation},
Ann. Inst. H. Poincar\'e C Anal. Non Lin\'eaire \textbf{36} (2019),
no.~2, 451--454.

\bibitem{Mooney}
C.~Mooney,
\emph{Strict \(2\)-convexity of convex solutions to the quadratic
Hessian equation},
Proc. Amer. Math. Soc. \textbf{149} (2021), no.~6, 2473--2477.

\bibitem{Qiu}
G.~Qiu,
\emph{Interior Hessian estimates for \(\sigma_2\) equations in
dimension three},
Front. Math. \textbf{19} (2024), no.~4, 577--598.

\bibitem{Savin}
O.~Savin,
\emph{Small perturbation solutions for elliptic equations},
Comm. Partial Differential Equations \textbf{32} (2007), 557--578.

\bibitem{ShankarYuanSemiconvex}
R.~Shankar and Y.~Yuan,
\emph{Hessian estimate for semiconvex solutions to the sigma-2
equation},
Calc. Var. Partial Differential Equations \textbf{59} (2020),
Art.~30.

\bibitem{ShankarYuanAlmostConvex}
R.~Shankar and Y.~Yuan,
\emph{Regularity for almost convex viscosity solutions of the
sigma-2 equation},
J. Math. Study \textbf{54} (2021), no.~2, 164--170.

\bibitem{ShankarYuanFour}
R.~Shankar and Y.~Yuan,
\emph{Hessian estimates for the sigma-2 equation in dimension four},
Ann. of Math. (2) \textbf{201} (2025), no.~2, 489--513.

\bibitem{SzekelyhidiWeinkove}
G.~Sz\'ekelyhidi and B.~Weinkove,
\emph{Weak Harnack inequalities for eigenvalues and constant rank
theorems},
Comm. Partial Differential Equations \textbf{46} (2021), no.~8,
1585--1600.

\bibitem{Urbas1990}
J.~I.~E. Urbas,
\emph{On the existence of nonclassical solutions for two classes of
fully nonlinear elliptic equations},
Indiana Univ. Math. J. \textbf{39} (1990), no.~2, 355--382.

\bibitem{WarrenYuan}
M.~Warren and Y.~Yuan,
\emph{Hessian estimates for the sigma-2 equation in dimension three},
Comm. Pure Appl. Math. \textbf{62} (2009), no.~3, 305--321.

\bibitem{Xu}
Y.~Xu,
\emph{Interior estimates for solutions of the quadratic Hessian
equations in dimension three},
J. Math. Anal. Appl. \textbf{489} (2020), Art.~124179.

\end{thebibliography}
\end{document}